\documentclass[11pt,a4paper]{article}
\usepackage{etex}
\usepackage[english]{babel}
\usepackage{amsmath,amsthm}
\usepackage{amsfonts}
\usepackage{amssymb}
\usepackage{accents}
\usepackage{pstricks-add}
\usepackage{graphicx}
\usepackage{color}
\usepackage{framed}
\usepackage{mathptmx}
\usepackage{enumerate}
\usepackage[arrow, matrix, curve]{xy}
\usepackage{pstricks}
\usepackage{pstricks, pstricks-add, mathptmx, amsfonts, hyperref, color}
\usepackage[left=3cm, right=3cm, top=3cm, bottom=3cm]{geometry}
\usepackage{calrsfs}
\usepackage{multicol}
\DeclareMathAlphabet{\pazocal}{OMS}{zplm}{m}{n}
\newtheorem{thm}{Theorem}[section]
\newtheorem{cor}[thm]{Corollary}
\newtheorem{lem}[thm]{Lemma}
\newtheorem{prop}[thm]{Proposition}
\theoremstyle{definition}
\newtheorem{defn}[thm]{Definition}
\newtheorem{rhp}[thm]{Riemann--Hilbert problem}
\theoremstyle{remark}
\newtheorem{rem}[thm]{Remark}
\numberwithin{equation}{section}

\newcommand{\set}[1]{\left\{#1\right\}}
\newcommand{\Real}[0]{\mathbb R}
\newcommand{\Compl}[0]{\mathbb C}
\newcommand{\eps}[0]{\varepsilon}

\newcommand{\re}{\,\mathfrak{Re}\,}
\newcommand{\im}{\,\mathfrak{Im}\,}

\newcommand{\tr}{\,\text{tr}\,}

\newcommand{\ii}{\mathrm{i}}

\def\res{\mathop{Res}}

\newcommand{\Pp}{\pazocal{P}^+}
\newcommand{\Pm}{\pazocal{P}^-}
\newcommand{\Ppm}{\pazocal{P}^{\pm}}
\newcommand{\rh}{Riemann--Hilbert problem }
\renewcommand\bf\bfseries
\begin{document}

\title{Global existence for the derivative NLS equation in the presence of solitons}
\author{Aaron Saalmann\thanks{A.S. gratefully acknowledges financial support from the projects ``Quantum Matter and Materials"
and SFB-TRR 191 ``Symplectic Structures in Geometry, Algebra and Dynamics" (Cologne University, Germany).}
\footnote{Mathematisches Institut, Universit\"{a}t zu K\"{o}ln, 50931 K\"{o}ln, Germany, e-mail: asaalman@math.uni-koeln.de}
}

\date{\today}
\maketitle
\begin{abstract}
   We prove the existence of global solutions to the DNLS equation with initial data in a large subset of $H^2(\Real)\cap H^{1,1}(\Real)$ containing a neighborhood of all solitons. We use the inverse scattering transform method, which was recently developed by D. Pelinovsky and Y. Shimabukuro, and an auto-B\"{a}cklund transform in order to include solitons.
\end{abstract}
\tableofcontents
\newpage
\section{Introduction}\label{s intro}
Consider the Cauchy problem for the derivative nonlinear Schr\"odinger (DNLS) equation
\begin{equation}\label{e dnls}
    \left\{
      \begin{array}{ll}
        \ii u_t+u_{xx}+\ii(|u|^2u)_x=0,\\
        u|_{t=0}=u_0,
      \end{array}
    \right.
\end{equation}
on $\Real$, where $u(x,t):\Real\times\Real\to\Compl$. Subscripts denote partial derivatives. In this paper we will prove the following global existence result:
\begin{thm}\label{t main}
    There exists an open subset $\pazocal{G}\subset H^{1,1}(\Real)$ such that if $u_0$ belongs to $H^2(\Real)\cap \pazocal{G}$, then there exists a unique global solution $u(\cdot,t)\in H^2(\Real)\cap \pazocal{G}$ of the Cauchy problem (\ref{e dnls}) for every $t\in\Real$.
\end{thm}
The spaces used in Theorem \ref{t main} are defined as follows.
\begin{equation*}
    H^k(\Real)=\set{u\in L^2(\Real),...,\partial_x^ku\in L^2(\Real)},\quad
    H^{1,1}(\Real)=\set{u\in L^{2,1}(\Real),\partial_xu\in L^{2,1}(\Real)},
\end{equation*}
where the weighted spaces $L^{2,s}(\Real)$ are defined by the norm
\begin{equation*}
    \|u\|_{L^{2,s}(\Real)}:=\left(\int_{\Real}\langle x\rangle^{2s}|u(x)|^2dx\right)^{1/2},\qquad  \langle x\rangle:=\sqrt{1+x^2}.
\end{equation*}
The question of global well-posedness of the DNLS equation (\ref{e dnls}) was an open problem for a long time. Local solvability in $H^s(\Real)$ with $s>3/2$ was shown in \cite{Tsutsumi1980}. Later in \cite{Tsutsumi1981}, the same authors presented a result on global solvability for $u_0\in H^2(\Real)$ under the assumption that the $H^1$ norm of $u_0$ is small. Similar global well-posedness results were proved in \cite{Hayashi1992,Hayashi1993},
where the authors work with $u_0\in H^1(\Real)$ and assume a small $L^2(\Real)$ norm. More than two decades later this upper bound on the $L^2(\Real)$ norm of the initial datum could be improved by \cite{Wu2013,Wu2015}. Only recently the authors of \cite{Hayashi2017} proved that there exist global solutions with any large $L^2(\Real)$ norm. They showed global existence of solutions for (\ref{e dnls}) with initial datum of the form $u_0=e^{icx}\psi$ where $\psi\in H^1(\Real)$ can be arbitrary and $c$ has to be chosen sufficiently large.
\medskip \\
None of the so far mentioned articles relies on the fact that the DNLS is formally solvable with the inverse scattering transform method. This structural property was discovered in \cite{KaupNewell1978}. The most extensive analysis of the Cauchy problem (\ref{e dnls}) using inverse scattering tools is certainly given by the series of papers \cite{Liu2016,LiuPerrySulem2017,LiuPerrySulem2017b}. In their first work the authors establish Lipschitz continuity of the direct and inverse scattering transform for the DNLS equation in appropriate function spaces and they prove global solvability for those initial data that are soliton-free. The second work is devoted to long-time behavior of solutions for soliton-free initial data. Therein it is proven that the amplitude of those solutions decays like $|t|^{-1/2}$ as $|t|\to\infty$. Using this dispersion result and including solitons the authors complete their studies in their third paper where they give a full description of the long-time behavior of the solutions. Moreover, the third paper contains a proof of global well-posedness under the same assumptions on the initial data as in our Theorem \ref{t main}. Other rigorous works on the inverse scattering transform in the context of the DNLS equation are given by \cite{Pelinovsky2016} (soliton-free case) and its complementing paper \cite{PelinShimaSaal2017} (finite number of eigenvalues). Whereas in \cite{Liu2016} a gauge equivalence of the DNLS with a related dispersive equation is used, in \cite{Pelinovsky2016} the direct scattering transformation is constructed for the DNLS equation itself. This technical difference leads to different spaces: $H^2(\Real)\cap H^{1,1}(\Real)$ is appropriate in \cite{Pelinovsky2016}, but in \cite{Liu2016} the space $H^{2,2}(\Real):=H^2(\Real)\cap L^{2,1}(\Real)$ is considered.
\medskip \\
In \cite{Liu2016,Pelinovsky2016} as well as in the present paper, the assumption $u_0\in H^2(\Real)\cap \pazocal{G}$ on the initial datum avoids resonances of the spectral problem (\ref{e Lax1}). But in contrast to the soliton-free case \cite{Liu2016,Pelinovsky2016}, the elements in $\pazocal{G}$ are allowed to admit eigenvalues of (\ref{e Lax1}). The set of eigenvalues
$\set{\lambda_1,...,\lambda_N}$ then corresponds to a particular multi-soliton, in whose neighborhood the solution $u(x,t)$ will be located. Since Theorem \ref{t main} is a natural extension of the main results in \cite{Liu2016,Pelinovsky2016} and, moreover, since our result is already covered by \cite{PelinShimaSaal2017} as well as by \cite{LiuPerrySulem2017b}, we cannot raise any claim of originality of the result itself. What makes this present paper new is the way how the existence of the inverse scattering map in the case of solitons is established. Whereas in \cite{LiuPerrySulem2017b} this technical issue is treated directly, we give a proof by adding successively more and more eigenvalues, see Lemma \ref{l solvability of RHP N=1}. For that purpose we use a B\"{a}cklund transformation found in \cite{Deift2011}, see (\ref{e Bäcklund for m^1}), and show that this transformation can be applied to the rigorous treatment of the DNLS equation. Technical statements such as Lemma \ref{l m(z0)-1 in H^11 and H^2} and Proposition \ref{p A inverse} become necessary and constitute the most original parts of our proof.

It shall be mentioned that a B\"{a}cklund transformation is also used in \cite{PelinShimaSaal2017}. But therein the transformation is applied directly to the solution $u$ in order to remove solitons.  Then by the solvability results from \cite{Liu2016,Pelinovsky2016} and the invertibility of the B\"{a}cklund transformation, the global well-posedness result follows. Hence, compared to  \cite{PelinShimaSaal2017}, the present paper does not only construct global solutions  of the DNLS equation for a large class of initial data but also solves the inverse scattering problem for those initial data.
\medskip \\
The paper is organized as follows. Section \ref{s scatt} contains the construction of the Jost functions and the definition of the scattering data for an initial datum $u_0\in H^2(\Real)\cap \pazocal{G}$. This section does not contain new results but follows closely \cite{Pelinovsky2016}.  At the end of Subsection \ref{ss scattering data} we formulate the Riemann--Hilbert problem as the starting point for the inverse scattering which is treated in Section \ref{s inverse sc} and \ref{s adding a pole}. For the convenience of the reader we inserted Section \ref{s solitons} where we shortly describe the phenomenon of solitary waves. Whereas Section \ref{s inverse sc} handles pure radiation solutions, in Section \ref{s adding a pole} we add a pole and obtain solutions in a neighborhood of a soliton. We split this procedure into two subsections since the cases $x>0$ and $x<0$ require different Riemann--Hilbert problems. Finally, in Section \ref{s proof} we use the local well-posedness theory in \cite{Tsutsumi1980} and \cite{Hayashi1992} and our estimates for the continuity of the inverse scattering to show that local solutions can be continued for all times.
\medskip \\
\\ 
\section{Direct scattering transform}\label{s scatt}
For a review of the scattering map for the DNLS equation we are going to follow closely \cite{Pelinovsky2016,Liu2016}. As pointed out in the pioneer work \cite{KaupNewell1978}, the DNLS equation is the compatibility condition for solutions $\psi\in\Compl^2$ of the linear system given by
\begin{equation}\label{e Lax1}
    \partial_x\psi=[-\ii\lambda^2\sigma_3+\lambda Q(u)]\psi
\end{equation}
and
\begin{equation}\label{e Lax2}
    \partial_t\psi=[-2\ii\lambda^4\sigma_3+ 2\lambda^3 Q(u)+\ii\lambda^2|u|^2\sigma_3- \lambda |u|^2Q(u) +\ii\lambda \sigma_3 Q(u_x)]\psi,
\end{equation}
where
\begin{equation*}
    Q(u)=
    \left[
      \begin{array}{cc}
        0 & u(x,t) \\
        -\overline{u}(x,t) & 0 \\
      \end{array}
    \right],\qquad
    \sigma_3=
    \left[
      \begin{array}{cc}
        1 & 0 \\
        0 & -1 \\
      \end{array}
    \right].
\end{equation*}
In this context the term \emph{compatibility condition} is chosen, because if the spectral parameter $\lambda$ is independent of $x$ and $t$, it can be shown that the formal equality of the mixed derivatives, $\partial_x\partial_t\psi=\partial_t\partial_x\psi$,  is equivalent to the statement that $u$ solves the DNLS equation (\ref{e dnls}).
\subsection{Jost functions}
It is natural to introduce solutions of (\ref{e Lax1}) which satisfy the same asymptotic behavior at infinity as solutions of the spectral problem (\ref{e Lax1}) in the case of vanishing potential $u\equiv 0$:
\begin{eqnarray*}
  \psi^{(-)}_1(\lambda;x)\sim
  \left(
    \begin{array}{c}
      1 \\
      0 \\
    \end{array}
  \right)e^{-\ii x\lambda^2}
  ,\qquad
  \psi^{(-)}_2(\lambda;x)\sim
  \left(
    \begin{array}{c}
      0 \\
      1 \\
    \end{array}
  \right)e^{\ii x\lambda^2}
  && \text{ as } x\to -\infty\\
  \psi^{(+)}_1(\lambda;x)\sim
  \left(
    \begin{array}{c}
      1 \\
      0 \\
    \end{array}
  \right)e^{-\ii x\lambda^2}
  ,\qquad
  \psi^{(+)}_2(\lambda;x)\sim
  \left(
    \begin{array}{c}
      0 \\
      1 \\
    \end{array}
  \right)e^{\ii x\lambda^2}
  && \text{ as } x\to+\infty.
\end{eqnarray*}
In order to have constant boundary conditions we introduce the \emph{normalized Jost functions} by
\begin{equation*}
    \varphi_{\pm}(\lambda;x)=\psi^{(\pm)}_1(\lambda;x) e^{\ii x\lambda^2},\qquad \phi_{\pm}(\lambda;x)=\psi^{(\pm)}_2(\lambda;x) e^{-\ii x\lambda^2},
\end{equation*}
such that we have
\begin{equation}\label{e asymptotics psi}
    \lim_{x\to\pm\infty}\varphi_{\pm}(\lambda;x)=e_1\quad\text{ and } \quad\lim_{x\to\pm\infty}\phi_{\pm}(\lambda;x)=e_2,
\end{equation}
where $e_1=(1,0)^T$ and $e_2=(0,1)^T$. The Jost functions are solutions of the following Volterra's integral equations
\begin{equation}\label{e volterra phi varphi}
    \begin{aligned}
        \varphi_{\pm}(\lambda;x)&=e_1 + \lambda \int_{\pm\infty}^x
        \left[
          \begin{array}{cc}
            1 & 0 \\
            0 & e^{2\ii\lambda^2(x-y)} \\
          \end{array}
        \right]
        Q(u(y))\varphi_{\pm}(\lambda;y)dy,\\
        \phi_{\pm}(\lambda;x)&=e_2 + \lambda \int_{\pm\infty}^x
        \left[
          \begin{array}{cc}
            e^{-2\ii\lambda^2(x-y)} & 0 \\
            0 & 1\\
          \end{array}
        \right]
        Q(u(y))\phi_{\pm}(\lambda;y)dy.
    \end{aligned}
\end{equation}
It can be shown that (\ref{e volterra phi varphi}) admit solutions $\varphi_-(\lambda;x)$ and $\phi_+(\lambda;x)$ for $\im(\lambda^2)>0$ and $\varphi_+(\lambda;x)$ and $\phi_-(\lambda;x)$ for $\im(\lambda^2)<0$. Moreover the dependence of $\lambda$ is analytic in the corresponding domains where the Jost functions exist.
However, due to the presence of $\lambda$ that multiplies the matrix $Q(u)$ in the linear equation (\ref{e Lax1}), standard fixed point arguments for (\ref{e volterra phi varphi}) are not uniform in $\lambda$. Therefore, in \cite{Pelinovsky2016} the authors worked out a transformation of the Kaup-Newell type spectral problem (\ref{e Lax1}) to a linear equation of the Zakharov-Shabat type. The idea of that kind of transformation can already be found in \cite{KaupNewell1978}. In what follows we are going to present this transformation and set
\begin{equation}\label{e def T1 Q1}
    T_1(\lambda;x)=
    \left[
      \begin{array}{cc}
        1 & 0 \\
        -\overline{u}(x) & 2\ii\lambda \\
      \end{array}
    \right],\qquad
    Q_1(u)=\frac{1}{2\ii}
    \left[
      \begin{array}{cc}
        |u|^2 & u \\
        -2\ii\overline{u}_x -\overline{u}|u|^2& -|u|^2 \\
      \end{array}
    \right],
\end{equation}
and
\begin{equation}\label{e def T2 Q2}
    T_2(\lambda;x)=
    \left[
      \begin{array}{cc}
        2\ii\lambda &-u(x) \\
        0 & 1 \\
      \end{array}
    \right],\qquad
    Q_2(u)=\frac{1}{2\ii}
    \left[
      \begin{array}{cc}
        |u|^2 & -2\ii u_x -u|u|^2 \\
        -\overline{u}& -|u|^2 \\
      \end{array}
    \right].
\end{equation}
Then, it is elementary to check that $z=\lambda^2$ and
\begin{equation}\label{e def M N}
    M_{\pm}(z;x)=T_1(\lambda;x)\varphi_{\pm} (\lambda;x),\quad
    N_{\pm}(z;x)=T_2(\lambda;x)\phi_{\pm}(\lambda;x)
\end{equation}
make (\ref{e volterra phi varphi}) equivalent to
\begin{equation}\label{e volterra M N}
    \begin{aligned}
        M_{\pm}(z;x)&=e_1 +  \int_{\pm\infty}^x
        \left[
          \begin{array}{cc}
            1 & 0 \\
            0 & e^{2\ii z(x-y)} \\
          \end{array}
        \right]
        Q_1(u(y))M_{\pm}(z;y)dy,\\
        N_{\pm}(z;x)&=e_2 +  \int_{\pm\infty}^x
        \left[
          \begin{array}{cc}
            e^{-2\ii z (x-y)} & 0 \\
            0 & 1\\
          \end{array}
        \right]
        Q_2(u(y))N_{\pm}(z;y)dy.
    \end{aligned}
\end{equation}
Note that the symmetries
\begin{equation}\label{e symmetries phi varphi 1}
    \varphi_{\pm}(\lambda;x)=
    \left[
      \begin{array}{cc}
        1 & 0 \\
        0 & -1 \\
      \end{array}
    \right]
    \varphi_{\pm}(-\lambda;x),\qquad
    \phi_{\pm}(\lambda;x)=
    \left[
      \begin{array}{cc}
        -1 & 0 \\
        0 & 1 \\
      \end{array}
    \right]
    \phi_{\pm}(-\lambda;x)
\end{equation}
make sure that (\ref{e def M N}) is well-defined.
Equations (\ref{e volterra M N}) are analogues to the integral equations known from the forward scattering for the NLS equation (see, e.g., \cite{Ablowitz2004}). If $Q_{1,2}(u)\in L^1(\Real)$, then, (\ref{e volterra phi varphi}) admit solutions $M_-(z;x)$ and $N_+(z;x)$ for $\im(z)>0$ and $M_+(z;x)$ and $N_-(z;x)$ for $\im(z)<0$. Moreover the dependence on $z$ is analytic in the corresponding domains where the Jost functions exist.
\begin{rem}
    The assumption $u\in H^{1,1}(\Real)$ in Theorem \ref{t main} is chosen such that $Q_{1,2}(u)\in L^1(\Real)$.
\end{rem}
Compared to (\ref{e volterra phi varphi}), in (\ref{e volterra M N}) there is no $\lambda$ which multiplies the integral. As a result, the Neumann series for (\ref{e volterra M N}) converge uniformly in $z$.
By means of the asymptotic expansion for large $z$ of the Jost functions, the potential $u$ can be reconstructed from $M_{\pm}$ and $N_{\pm}$, respectively (see \cite[Lemma 2]{Pelinovsky2016}). Furthermore, regularity properties of $M_{\pm}$ and $N_{\pm}$ are used in \cite{Pelinovsky2016} to prove regularity of the reflection coefficient $r_+$ and $r_-$ which we will define in (\ref{e def r pm}) in the next subsection on the Scattering data.
\subsection{Scattering data}\label{ss scattering data}
We recall that $\varphi_{\pm}(\lambda;x)e^{-\ii x\lambda^2}$ and $\phi_{\pm}(\lambda;x)e^{\ii x\lambda^2}$ are solutions of the spectral problem (\ref{e Lax1}) with boundary condition (\ref{e asymptotics psi}). Taking into account $\tr(\sigma_3)=\tr(Q)=0$ we find
\begin{equation}\label{e det phi psi =1}
    \lim_{x\to\pm\infty}\det[\varphi_{\pm} (\lambda;x)e^{-\ii x\lambda^2},\phi_{\pm}(\lambda;x)e^{+\ii x\lambda^2}]=1
\end{equation}
for all $\lambda^2\in\Real$ and $x\in\Real$. Thus, in particular $\varphi_{+}e^{-\ii x\lambda^2}$ and $\phi_{+}e^{\ii x\lambda^2}$ are linearly independent and by ODE theory they form a basis of the space of solutions of the spectral problem (\ref{e Lax1}). This enables us to express the "$-$" Jost functions in terms of the "$+$" Jost functions for every $\lambda^2\in\Real$ and $x\in\Real$. According to that, there exist coefficients $\alpha,\beta,\gamma,\delta$ which satisfy:
\begin{equation}\label{e scattering relation}
    \begin{aligned}
         \varphi_{-}(\lambda;x)e^{-\ii x\lambda^2}&&=&& \alpha(\lambda)\varphi_{+}(\lambda;x)e^{-\ii x\lambda^2} &&+&&\beta(\lambda) \phi_{+}(\lambda;x)e^{\ii x\lambda^2},\\
         \phi_{-}(\lambda;x)e^{\ii x\lambda^2}&&=&& \gamma(\lambda)\varphi_{+}(\lambda;x)e^{-\ii x\lambda^2} &&+&& \delta(\lambda) \phi_{+}(\lambda;x)e^{\ii x\lambda^2}.
    \end{aligned}
\end{equation}
The matrix
$
\left[
  \begin{array}{cc}
    \alpha & \beta \\
    \gamma & \delta \\
  \end{array}
\right]
$
is referred to as the \emph{transfer matrix} in the literature and (\ref{e scattering relation}) is called \emph{scattering relation}. By (\ref{e det phi psi =1}), we verify that the determinant of the transfer matrix equals one. By Cramer's rule we find
\begin{equation}\label{e def alpha beta}
    \begin{aligned}
        \alpha(\lambda)&=\det[\varphi_{-}(\lambda;x), \phi_{+}(\lambda;x)],\\[2pt]
        \beta(\lambda)&=\det[\varphi_{+} (\lambda;x)e^{-\ii x\lambda^2}, \varphi_{-}(\lambda;x)e^{-\ii x\lambda^2}].
    \end{aligned}
\end{equation}
Making again use of $\tr(\sigma_3)=\tr(Q)=0$, we justify that $\alpha$ and $\beta$ indeed do not depend on $x$. Moreover $\alpha$ can be analytically extended to the first and third quadrant, where $\im(\lambda^2)>0$, which follows from the analytic properties of the Jost functions $\varphi_{-}$, $\phi_{+}$ in this domain. Furthermore, from the symmetry
\begin{equation}\label{e symmetries phi varphi 2}
    \phi_{\pm}(\overline{\lambda};x)=
    \left[
      \begin{array}{cc}
        0 & -1 \\
        1 & 0 \\
      \end{array}
    \right]\overline{\varphi_{\pm}(\lambda;x)},
\end{equation}
which are direct consequences of integral equations (\ref{e volterra phi varphi}), we can derive from the scattering relation (\ref{e scattering relation}) the following conservation law:
\begin{equation}\label{e alpha^2+beta^2}
    \left\{
      \begin{array}{ll}
        |\alpha(\lambda)|^2+|\beta(\lambda)|^2=1, \quad\lambda\in\Real,\\
        |\alpha(\lambda)|^2-|\beta(\lambda)|^2=1, \quad\lambda\in\ii\Real.
      \end{array}
    \right.
\end{equation}
As pointed out in \cite{Pelinovsky2016} this is indicating that the DNLS equation combines elements of the focusing and as well of the defocusing cubic NLS equation.

We now continue with the definition of the reflection coefficient:
\begin{equation}\label{e def r}
    r(\lambda)=\frac{\beta(\lambda)}{\alpha(\lambda)}.
\end{equation}
This definition makes sense for every $\lambda^2\in\Real$, if $\alpha$ admits no zeros on $\Real\cup\ii\Real$, but we can not expect generally that $\alpha$ behaves like that. Therefore we define the following set:
\begin{equation}\label{e def no resonances}
    \pazocal{R}:=\set{u\in H^{1,1}(\Real):\exists A>0,|\alpha(\lambda)|>A\text{ for every }\lambda\in\Real\cup\ii\Real}
\end{equation}
Zeroes $\lambda\in\Real\cup\ii\Real$ of $\alpha$ are called \emph{resonances} in \cite{Pelinovsky2016}. Hence, the set $\pazocal{R}$ consists of those potentials, which do not admit resonances of the linear equation (\ref{e Lax1}). Let us assume from now on that $u\in \pazocal{R}$. Then,  we can rewrite the scattering relation (\ref{e scattering relation}) in the following way:
\begin{equation}\label{e alternative scattering relation}
    \Phi_+(\lambda;x)=\Phi_-(\lambda;x)(1+S(\lambda;x)), \quad\lambda^2\in\Real,
\end{equation}
where the matrices $\Phi_{\pm}$ and $S$ are given by
\begin{equation}\label{e def Phi}
    \Phi_+(\lambda;x):=
    \left[\frac{\varphi_{-}(\lambda;x)} {\alpha(\lambda)},\phi_{+}(\lambda;x)
    \right],\quad
    \Phi_-(\lambda;x):=
    \left[\varphi_{_+}(\lambda;x) ,\frac{\phi_{-}(\lambda;x)} {\overline{\alpha(\overline{\lambda})}}
    \right],
\end{equation}
and
\begin{equation}\label{e def S}
    S(\lambda;x):=
    \left\{
      \begin{array}{ll}
        \left[
          \begin{array}{cc}
            |r(\lambda)|^2 & \overline{r(\lambda)} e^{-2\ii x\lambda^2} \\
            r(\lambda) e^{2\ii x\lambda^2} & 0 \\
          \end{array}
        \right]
        , & \hbox{for }\lambda\in\Real,\vspace{2mm} \\ \left[
          \begin{array}{cc}
            -|r(\lambda)|^2 & -\overline{r(\lambda)} e^{-2\ii x\lambda^2} \\
            r(\lambda) e^{2\ii x\lambda^2} & 0 \\
          \end{array}
        \right]
        , & \hbox{for }\lambda\in\ii\Real.
      \end{array}
    \right.
\end{equation}
It is clear from the representation (\ref{e def alpha beta}) that $\alpha$ has an analytic continuation in the first and third quadrants of the $\lambda$ plane. Therefore the function $\Phi_+$ defined in (\ref{e def Phi}) can be continued analytically in the first and third quadrants, as long as there are no zeros $\lambda_0$ of the continuation of $\alpha$ with $\im(\lambda_0^2)>0$. Under the same assumption, the function $\Phi_-$ in (\ref{e def Phi}) can be analytically continued in the second and fourth quadrant. From now on we want to allow that $\alpha(\lambda)$ has finite many simple zeroes. That is $\alpha(\lambda_k)=0$ and $\alpha'(\lambda_k)\neq0$ for a finite number of pairwise different $\lambda_1,...,\lambda_N$ which are assumed to lie in the first quadrant. Note that, if $\alpha(\lambda_k)=0$, then also $\alpha(-\lambda_k)=0$. Henceforth, the continuations of $\Phi_{\pm}$ are merely meromorphic. They admit simple poles at the zeros of $\alpha$, since  $\alpha'(\lambda_k)\neq0$ for $k=1,...,N$. The prime denotes the derivative with respect to $\lambda$. We find:
\begin{equation*}
    \res_{\lambda=\pm\lambda_k}
    \Phi_+(\lambda;x)=
    \left[\frac{\varphi_{-}(\pm\lambda_k;x)} {\pm\alpha'(\lambda_k)},\;0\;
    \right].
\end{equation*}
By (\ref{e def alpha beta}), the meaning of the zeros of $\alpha$ is the following. If $\alpha(\lambda_k)=0$, then by (\ref{e def alpha beta}) the $\Compl^2$ vectors $\varphi_{-}(\lambda_k;x)e^{-\ii x\lambda_k^2}$ and $\phi_{+}(\lambda_k;x)e^{\ii x\lambda_k^2}$ are linear dependent for every $x\in\Real$. Hence,
\begin{equation}\label{e def gamma}
    \varphi_{-}(\pm\lambda_k;x)=\pm\gamma_k\; e^{2\ii x\lambda_k^2}\; \phi_{+}(\pm\lambda_k;x)
\end{equation}
for some complex constant $\gamma_k\in\Compl\setminus\set{0}$. We will refer to $\gamma_k$ as the \emph{norming constant}. The norming constants do not depend on $x$. Indeed, differentiating (\ref{e def gamma}) with respect to $x$ and using the fact that  $\varphi_{-}(\lambda_k;x)e^{-\ii x\lambda_k^2}$ and $\phi_{+}(\lambda_k;x)e^{\ii x\lambda_k^2}$ are solutions of the spectral problem (\ref{e Lax1}), we easily obtain $\partial_x \gamma_k=0$. Note also that due to the symmetry (\ref{e symmetries phi varphi 1}) the cases $+\lambda_k$ and $-\lambda_k$ do have the same norming constants upon a minus sign.  Combining (\ref{e def gamma}) and the above residue calculation we find
\begin{equation}\label{e res Phi}
    \res_{\lambda=\pm\lambda_k}
    \Phi_+(\lambda;x)=
    \left[\frac{\pm\gamma_k \; e^{2\ii x\lambda_k^2}} {\alpha'(\pm\lambda_k)}\phi_{+}(\pm\lambda_k;x),\;0\;
    \right]=
    \lim_{\lambda\to\pm\lambda_k}\Phi_+(\lambda;x)
    \left[
      \begin{array}{cc}
        0 & 0\\
        \frac{\gamma_k \; e^{2\ii x\lambda_k^2}} {\alpha'(\lambda_k)} & 0 \\
      \end{array}
    \right].
\end{equation}
Correspondingly, we can compute an analogue relation for the residue of $\Phi_-$ at $\pm\overline{\lambda}_k$.\\
By a theorem of complex analysis (see, e.g., \cite[Theorem 3.2.8]{Ablowitz2003}), the zeroes of $\alpha$ must be isolated. In addition, by \cite[Lemma 4]{Pelinovsky2016} we know $\alpha(\lambda)\to\alpha_{\infty}\neq 0$ as $|\lambda|\to\infty$. Thus, we conclude that the zeroes of $\alpha(\lambda)$ in the first quadrant form a finite set $\set{\lambda_1,...,\lambda_N}$. But the essential assumption $\alpha'(\lambda_k)\neq 0$ is generally not expectable and give rise to the following definition:
\begin{equation}\label{e def no eigenvalues}
    \pazocal{E}:=\set{u\in H^{1,1}(\Real):\alpha'(\lambda_k)\neq 0\text{ for all zeroes }\lambda_k\text{ of }\alpha\text{ with } \im(\lambda_k^2)>0}
\end{equation}
From now on, additionally to $u\in\pazocal{R}$,  we assume $u\in\pazocal{G}:=\pazocal{R}\cap\pazocal{E}$. The elements of $\pazocal{G}$ are called \emph{generic potentials} according to the classical paper \cite{Beals1984}. As remarked by the authors in \cite[Remark 5]{Pelinovsky2016}, we have $u\in\pazocal{G}$ if
\begin{equation*}\label{e estimate for no eigenv and reso}
    \|u\|_{L^2}^2+\sqrt{\|u\|_{L^1}( 2\|\partial_x u\|_{L^1}+\|u\|_{L^3}^3)}<1.
\end{equation*}
The set $\pazocal{G}$ is open and, moreover, dense in $H^{1,1}(\Real)$. Due to the availability of the transformation (\ref{e def M N}), this can be deduced from \cite{Beals1984} as explained in \cite[Proposition 4]{PelinShimaSaal2017}. However, any soliton or multi soliton is contained in $\pazocal{G}$. For those explicit solutions, the expression
\begin{equation*}
    \|u\|_{L^2}^2+\sqrt{\|u\|_{L^1}( 2\|\partial_x u\|_{L^1}+\|u\|_{L^3}^3)}
\end{equation*}
can be arbitrary large.
\medskip\\
Using the transformation (\ref{e def M N}) it is shown in \cite{Pelinovsky2016} that for $u\in H^2(\Real)\cap H^{1,1}(\Real)$ the following holds.
\begin{equation}\label{e def Phi infty}
    \Phi_{\pm}(\lambda;x)\to\Phi_{\infty}(x):=
    \left[
      \begin{array}{cc}
        e^{-\frac{1}{2\ii}\int^{+\infty}_x|u(y)|^2dy} & 0 \\
        0 & e^{-\frac{1}{2\ii}\int_{-\infty}^x|u(y)|^2dy} \\
      \end{array}
    \right]\quad\text{as }|\lambda|\to\infty.
\end{equation}
The limit has to be taken along a contour  in the corresponding domain of analyticity.

The alternative scattering relation (\ref{e alternative scattering relation}), the residue condition (\ref{e res Phi}) and finally the asymptotic behavior (\ref{e def Phi infty}) set up a \rh.
Since that \rh is somewhat unsuitable to show the existence of the inverse Scattering map, we turn again to the Zhakarov-Shabat type Jost functions $M_{\pm}$ and $N_{\pm}$ (see (\ref{e def M N}), which are functions of $z$, where we recall $z=\lambda^2$. Due to $\alpha(\lambda)=\alpha(-\lambda)$, it is alowed to define $a(z):=\alpha(\lambda)$. Of course, if $\pm\lambda_k\neq0$ are (simple) zeroes of $\alpha$, then $z_k:=\lambda_k^2$ is a (simple) zero of $a$. In order to transfer the jump condition (\ref{e alternative scattering relation}) to the Jost functions $M_{\pm}$ and $N_{\pm}$, one more explicit definition is needed:
\begin{equation}\label{e def P}
    P_{\pm}(z;x):=\frac{1}{2\ii\lambda}
    T_1(\lambda;x)T_2^{-1}(\lambda;x)N_{\pm}(z;x)=
    -\frac{1}{4z}
    \left[
      \begin{array}{cc}
        1 & u(x) \\
        -\overline{u}(x) & -|u(x)|^2-4z \\
      \end{array}
    \right]N_{\pm}(z;x).
\end{equation}
In \cite[Lemma 5]{Pelinovsky2016} it is shown, that there is no singularity in (\ref{e def P}) and moreover, $P_{\pm}(z;x)$ satisfy the following limits as $|\im(z)|\to\infty$ along a contour in the domains of their analyticity:
\begin{equation*}
    \lim_{|z|\to\infty}P_{\pm}(z;x)=
    \left(
      \begin{array}{c}
        0\\
        N_{\pm}^{\infty}(x) \\
      \end{array}
    \right).
\end{equation*}
Now we are ready to define the analogue of (\ref{e def Phi}). Instead of $\lambda\in\Real\cup\ii\Real$, now we have $z\in\Real$ and set
\begin{equation}\label{e def pi}
    \pi_+(z;x):=
    \left[\frac{M_{-}(z;x)} {a(z)},P_{+}(z;x)
    \right],\quad
    \pi_-(z;x):=
    \left[M_{_+}(z;x) ,\frac{P_{-}(z;x)} {\overline{a(\overline{z})}}
    \right].
\end{equation}
These definitions entail the following analogue of (\ref{e alternative scattering relation}) which can be checked by elementary calculations:
\begin{equation}\label{e jump of pi}
    \pi_+(z;x)=\pi_-(z;x)(1+R(z;x)), \quad z\in\Real.
\end{equation}
Herein the new jump matrix $R$ which includes new reflection coefficients $r_{\pm}$, is defined by
\begin{equation*}
    R(z;x):=
        \left[
          \begin{array}{cc}
            \overline{r}_+(z)r_-(z) & e^{-2\ii xz}\overline{r}_+(z) \\
            e^{2\ii xz}r_-(z) & 0 \\
          \end{array}
        \right].
\end{equation*}
The new reflection coefficients are given by
\begin{equation}\label{e def r pm}
    r_{+}(z):=-\frac{\beta(\lambda)} {2\ii\lambda\alpha(\lambda)},\quad
    r_{-}(z):=\frac{2\ii\lambda\beta(\lambda)} {\alpha(\lambda)},\quad z\in\Real.
\end{equation}
We have the following Lemma \cite{Pelinovsky2016}.
\begin{lem}\label{l r pm in H1 and L^21}
    If $u\in H^2(\Real)\cap H^{1,1}(\Real)\cap\pazocal{R}$, then $r_{\pm}\in H^1(\Real) \cap L^{2,1}(\Real)$.
\end{lem}
Moreover, we found directly from the definition (\ref{e def r pm}) that $r_+$ and $r_-$ are connected by
\begin{equation}\label{e relation r+ r-}
    r_-(z)=4zr_+(z),\quad z\in\Real.
\end{equation}
Furthermore, $\overline{r}_+(z)r_-(z)=|r(\lambda)|^2$ if $z>0$, whereas $\overline{r}_+(z)r_-(z)=-|r(\lambda)|^2$ if $z<0$. Additionally, using (\ref{e alpha^2+beta^2}) we obtain $1-|r(\lambda)|^2=|\alpha(\lambda)|^{-2}$. Thus, we have
\begin{equation}\label{e r constraint}
   \left\{
     \begin{array}{ll}
       1+\overline{r}_+(z)r_-(z)\geq1, & z>0, \\
       1+\overline{r}_+(z)r_-(z)\geq c_0^2, & z<0,
     \end{array}
   \right.
\end{equation}
where $c_0^{-1}:=\sup_{\lambda\in\ii\Real}|\alpha(\lambda)|$. The constraint (\ref{e r constraint}) is used in \cite{Pelinovsky2016} to obtain a unique solution to the \rh \ref{rhp m} below.\par
Analytic continuations of $\pi_{\pm}$ in $\Compl^{\pm}$ exist if there is no $z\in\Compl$ such that $a(z)=0$. Otherwise we have analogously to (\ref{e def gamma})
\begin{equation*}
    M_{-}(z_k;x)=2\ii\lambda_k\gamma_k\; e^{2\ii x z_k}\; P_{+}(z_k;x)
\end{equation*}
with the same $\gamma_k$ as in (\ref{e def gamma}). Denoting the meromorphic continuations of $\pi_{\pm}(\cdot;x)$ with the same letters we can verify the following residue condition:
\begin{equation}\label{e res pi}
    \res_{z=z_k}\pi_+(z;x)=\lim_{z\to z_k}\pi_+(z;x)
          \left[
            \begin{array}{cc}
              0 & 0 \\
              2\ii\lambda_kc_k e^{2\ii xz_k} & 0
            \end{array}
          \right],
\end{equation}
where we set $c_k:=\gamma_k/a'(z_k)$. Correspondingly we can compute an analogue relation for the residuum of $\pi_-$ at $\overline{z}_k$. Next, we have
\begin{equation*}
    \pi_{\pm}(\lambda;x)\to\Phi_{\infty}(x)\quad\text{as }|\lambda|\to\infty,
\end{equation*}
similarly to (\ref{e def Phi infty}).
We obtain our final Riemann--Hilbert problem if we normalize the boundary condition at infinity:
\begin{equation}\label{e def m}
    m(z;x):=
    \left\{
      \begin{array}{ll}
        \,[\Phi_{\infty}(x)]^{-1}\pi_+(z;x), & z\in\Compl^+, \\
        \,[\Phi_{\infty}(x)]^{-1}\pi_-(z;x), & z\in\Compl^-.
      \end{array}
    \right.
\end{equation}
The multiplication from the left by the diagonal matrix $[\Phi_{\infty}(x)]^{-1}$ changes neither the analytic properties of $\pi_{\pm}$ nor the jump or residuum conditions. Therefore, the function $m$ defined in (\ref{e def m}) solves the following Riemann--Hilbert problem:
\begin{samepage}
\begin{framed}
\begin{rhp}\label{rhp m}
Find for each $x\in\Real$ a $2\times 2$-matrix valued function $\Compl\ni z\mapsto m(z;x)$ which satisfies
\begin{enumerate}[(i)]
  \item $m(z;x)$ is meromorphic in $\Compl\setminus\Real$ (with respect to the parameter $z$).
  \item $m(z;x)=1+\mathcal{O}\left(\frac{1}{z}\right)$ as $|z|\to\infty$.
  \item The non-tangential boundary values $m_{\pm}(z;x)$ exist for $z\in\Real$ and satisfy the jump relation
      \begin{equation}\label{e jump}
        m_+=m_-(1+R),\quad\text{where }
        R(z;x):=
        \left[
          \begin{array}{cc}
            \overline{r}_+(z)r_-(z) & e^{-2\ii xz}\overline{r}_+(z) \\
            e^{2\ii xz}r_-(z) & 0 \\
          \end{array}
        \right]
      \end{equation}
  \item $m$ has simple poles at $z_1,...,z_N,\overline{z}_1,...,\overline{z}_N$ with
      \begin{equation*}
        \begin{aligned}
          \res_{z=z_k}m(z;x)&=\lim_{z\to z_k}m(z;x)
          \left[
            \begin{array}{cc}
              0 & 0 \\
              2\ii\lambda_kc_k e^{2\ii xz_k} & 0
            \end{array}
          \right],\\
          \res_{z=\overline{z_k}}m(z;x)&=\lim_{z\to \overline{z}_k}m(z;x)
          \left[
            \begin{array}{cc}
              0 & \frac{-\overline{c}_k}{2\ii\lambda_k} e^{-2\ii x\overline{z}_k} \\
              0 & 0
            \end{array}
          \right].
        \end{aligned}
      \end{equation*}
\end{enumerate}
\end{rhp}
\end{framed}
\end{samepage}
We will use the notation
\begin{equation*}
    \mathcal{S}(u)=\set{r_{\pm};\lambda_1,... ,\lambda_N;c_1,...,c_N}
\end{equation*}
and call $\mathcal{S}$ the scattering data of $u$. They consist of the \emph{reflection coefficients} $r_{\pm}$ which satisfy the constraints (\ref{e relation r+ r-}) and (\ref{e r constraint}), the \emph{poles} $z_k:=\lambda_k^2$ and the \emph{norming constants} $c_k=\gamma_k/a'(z_k)$. $\mathcal{S}$ is all information we need to know about $u$ to formulate the \rh\ref{rhp m}. In the rest of this paper we treat the problem to define the inverse map $\set{r_{\pm};\lambda_1,... ,\lambda_N;c_1,...,c_N}\mapsto u$. Therefore we will solve \rh \ref{rhp m} and apply the following reconstruction formulas:
\begin{equation}\label{e rec 1}
    u(x)e^{\ii\int_{+\infty}^x|u(y)|^2dy}= -4\lim_{|z|\to\infty}z\;[m(z;x)]_{12}
\end{equation}
and
\begin{equation}\label{e rec 2}
    e^{-\frac{1}{2\ii}\int_{+\infty}^x|u(y)|^2dy} \partial_x\left(\overline{u}(x) e^{\frac{1}{2\ii}\int_{+\infty}^x|u(y)|^2dy}\right) =2\ii\lim_{|z|\to\infty}z\;[m(z;x)]_{21}.
\end{equation}
Both, (\ref{e rec 1}) and (\ref{e rec 2}), are justified in \cite{Pelinovsky2016} and the key of Inverse Scattering. By $[\cdot]_{ij}$ we denote the $i$-$j$-component of the matrix in the brackets.\\
The miraculous fact about the forward scattering is the trivial time evolution of the scattering data if the potential $u(x,t)$ evolves accordingly to the DNLS equation:
\begin{lem}\label{l time dependence scattering data}
    Under the assumption that an initial datum $u_0\in H^2(\Real)\cap H^{1,1}(\Real)\cap \pazocal{G}$ admits a (local) solution $u(\cdot,t)\in H^2(\Real)\cap H^{1,1}(\Real)$ to the Cauchy problem (\ref{e dnls}) for $t\in[0,T]$, the scattering data of $u(\cdot,t)$ are given by
    \begin{equation}\label{e time dependence scattering data}
        \mathcal{S}_t(u)=\set{r_{\pm}(z;t)=r_{\pm}(z;0) e^{4\ii z^2t};\lambda_1,... ,\lambda_N;c_1(0) e^{4\ii \lambda_1^4t},...,c_N(0) e^{4\ii \lambda_N^4t}},
    \end{equation}
    where
    \begin{equation*}
        \mathcal{S}_0(u)=\set{r_{\pm}(z;0);\lambda_1,... ,\lambda_N;c_1(0),...,c_N(0)}
    \end{equation*}
    are defined to be the scattering data of $u_0$. In particular, the set $\pazocal{G}$ is invariant under the flow of the DNLS equation, $r_{\pm}(\cdot;t)\in H^1(\Real)\cap L^{2,1}(\Real)$ for every $t\in[0,T]$, and, finally, (\ref{e relation r+ r-}) and (\ref{e r constraint}) remain valid.
\end{lem}
The proof of this Lemma is given in \cite[Section 5]{Pelinovsky2016} and we skip it here. Plugging the time dependence (\ref{e time dependence scattering data}) into the formulas of \rh \ref{rhp m} we obtain the dynamic Riemann--Hilbert problem for the DNLS equation.
\begin{samepage}
\begin{framed}
\begin{rhp}\label{rhp m dynamic}
Find for each $(x,t)\in\Real\times\Real$ a $2\times 2$-matrix valued function $\Compl\ni z\mapsto m(z;x,t)$ which satisfies
\begin{enumerate}[(i)]
  \item $m(z;x,t)$ is meromorphic in $\Compl\setminus\Real$ (with respect to the parameter $z$).
  \item $m(z;x,t)=1+\mathcal{O}\left(\frac{1}{z}\right)$ as $|z|\to\infty$.
  \item The non-tangential boundary values $m_{\pm}(z;x,t)$ exist for $z\in\Real$ and satisfy the jump relation
      \begin{equation*}
        m_+=m_-(1+R),\quad\text{where }
        R(z;x,t):=
        \left[
          \begin{array}{cc}
            \overline{r}_+(z)r_-(z) & e^{\overline{\phi}(z)}\overline{r}_+(z) \\
            e^{\phi(z)}r_-(z) & 0 \\
          \end{array}
        \right]
      \end{equation*}
      with $\phi(z):=2\ii xz+4\ii z^2t$.
  \item $m$ has simple poles at $z_1,...,z_N,\overline{z}_1,...,\overline{z}_N$ with
      \begin{equation}\label{e Res}
        \begin{aligned}
          \res_{z=z_k}m(z;x,t)&=\lim_{z\to z_k}m(z;x,t)
          \left[
            \begin{array}{cc}
              0 & 0 \\
              2\ii\lambda_kc_k e^{\phi_k} & 0
            \end{array}
          \right],\\
          \res_{z=\overline{z}_k}m(z;x,t)&=\lim_{z\to \overline{z}_k}m(z;x,t)
          \left[
            \begin{array}{cc}
              0 & \frac{-\overline{c}_k}{2\ii\lambda_k} e^{\overline{\phi}_k} \\
              0 & 0
            \end{array}
          \right],
        \end{aligned}
      \end{equation}
      where $\phi_k:=\phi(z_k)$.
\end{enumerate}
\end{rhp}
\end{framed}
\end{samepage}
\begin{rem}\label{r uniqueness + det=1}
    Without further theory we can observe that if \rh \ref{rhp m dynamic} is solvable, then the solution is unique. In order to show the uniqueness of solutions, we firstly find the following (trivial) Riemann Hilbert problem for the map $z\mapsto \det(m(z;x,t))$:
    \begin{equation*}
        \left\{
           \begin{array}{ll}
             \det(m(z;x,t))\text{ is an entire function with respect to the parameter }z,\\
             \det(m(z;x,t))\to 1,\text{ as }|z|\to\infty.
           \end{array}
         \right.
    \end{equation*}
    By Liouville's theorem we conclude
    \begin{equation}\label{e det m=1}
        \det(m(z;x,t))\equiv 1,\text{ for all }x,t\in\Real\text{ and }z\in \Compl.
    \end{equation}
    Hence, for a possible solution $m$ of \rh \ref{rhp m dynamic}, $[m(z;x,t)]^{-1}$ exists for all $x\in\Real$ and $z\in \Compl$. If we have a second solution $\widetilde{m}(z;x,t)$, the ratio $\widetilde{m}(z;x,t)[m(z;x,t)]^{-1}$ satisfies
    \begin{equation*}
        \left\{
           \begin{array}{ll}
             \widetilde{m}(z;x,t)[m(z;x,t)]^{-1}\text{ is an entire function with respect to the parameter }z,\\
             \widetilde{m}(z;x,t)[m(z;x,t)]^{-1}\to 1,\text{ as }|z|\to\infty,
           \end{array}
         \right.
    \end{equation*}
    such that $\widetilde{m}(z;x,t)[m(z;x,t)]^{-1}\equiv 1$.
\end{rem}
We end the subsection mentioning the following symmetry:
\begin{equation}\label{e symmetrie of m}
    m(z;x)=\frac{1}{4z}
    \left[
      \begin{array}{cc}
        w(x) & 1 \\
        -|w(x)|^2-4z & \overline{w}(x) \\
      \end{array}
    \right]
    \overline{m(\overline{z};x)}
    \left[
      \begin{array}{cc}
        0 & 1 \\
        4z & 0 \\
      \end{array}
    \right],
\end{equation}
where $w(x):= u(x)\,e^{\ii\int^x_{+\infty}|u(y)|^2dy}$. The symmetry (\ref{e symmetrie of m}) is obtained when one transfers the symmetry (\ref{e symmetries phi varphi 2}) to $\pi_{\pm}$ and $m$, respectively. 

\section{Solitons}\label{s solitons}
This section is devoted to the exact solitary wave solutions of the DNLS equation (\ref{e dnls}) which are known since the 1970s (see, e.g., \cite{mjlhus1976} and \cite{KaupNewell1978}). Also more recent works are concerned with solitons. See for instance \cite{Colin2006}, where orbital stability of solitons is shown. The inverse scattering machinery admits a simple definition of $N$-solitons:
\begin{defn}
    (Global) solutions $u^{(N\text{-sol})}(x,t)$ of (\ref{e dnls}) such that the initial datum $u^{(N\text{-sol})}(\cdot,0)$ produces scattering data
    \begin{equation*}
        \mathcal{S} (u^{(N\text{-sol})})=\set{r_+\equiv r_-\equiv 0;\lambda_1,... ,\lambda_N;c_1,...,c_N},
    \end{equation*}
    are called $N$-\emph{solitons}. For $N=1$ we just say \emph{soliton}.
\end{defn}
In the case of $r_+\equiv r_-\equiv 0,$ the Riemann--Hilbert problem \rh \ref{rhp m dynamic} reads as follows:
\begin{samepage}
\begin{framed}
\begin{rhp}\label{rhp N sol}
Find for each $x\in\Real$ a $2\times 2$-matrix valued function $\Compl\ni z\mapsto m^{(N\text{-sol})} (z;x,t)$ which satisfies
\begin{enumerate}[(i)]
  \item $m^{(N\text{-sol})} (z;x,t)$ is meromorphic in $\Compl$ (with respect to the parameter $z$).
  \item $m^{(N\text{-sol})} (z;x,t)=1+\mathcal{O}\left(\frac{1}{z}\right)$ as $|z|\to\infty$.
  \item $m^{(N\text{-sol})} $ has simple poles at $z_1,...,z_N,\overline{z}_1,...,\overline{z}_N$ with
      \begin{equation*}
        \begin{aligned}
          \res_{z=z_k}m^{(N\text{-sol})} (z;x,t)&=\lim_{z\to z_k}m^{(N\text{-sol})} (z;x,t)
          \left[
            \begin{array}{cc}
              0 & 0 \\
              2\ii\lambda_kc_k e^{2\ii xz_k+4\ii t z_k^2} & 0
            \end{array}
          \right],\\
          \res_{z=\overline{z}_k}m^{(N\text{-sol})} (z;x,t)&=\lim_{z\to \overline{z}_k}m^{(N\text{-sol})} (z;x,t)
          \left[
            \begin{array}{cc}
              0 & \frac{-\overline{c}_k}{2\ii\lambda_k} e^{-2\ii x\overline{z}_k-4\ii t \overline{z}_k^2} \\
              0 & 0
            \end{array}
          \right].
        \end{aligned}
      \end{equation*}
\end{enumerate}
\end{rhp}
\end{framed}
\end{samepage}
Using the ansatz
\begin{equation*}
    m^{(N\text{-sol})} (z;x,t)=1+\sum_{k=1}^N\left\{ \frac{A_k(x,t)}{z-z_k}+ \frac{B_k(x,t)}{z-\overline{z}_k}\right\}
\end{equation*}
we can transfer \rh \ref{rhp N sol} into a purely algebraic system which can be solved explicitly. Then, the reconstruction formulas (\ref{e rec 1}) and (\ref{e rec 2}) yield explicit solutions of the DNLS equation, which are (multi) solitons. For the special case $N=1$ we find
\begin{equation}\label{e soliton}
    u_{\omega,v,x_0,\gamma}(x,t)=\phi_{\omega,v}(x-vt-x_0) e^{-\ii\gamma+\ii\omega t+\ii\frac{v}{2}(x-vt)- \frac{3}{4}\ii\int_{\infty}^{x-vt-x_0} |\phi_{\omega,v}(y)|^2dy},
\end{equation}
where
\begin{equation}\label{e sol ampl}
    \phi_{\omega,v}(x)=\left[\frac{\sqrt{\omega}}{4\omega-v^2} \left\{\cosh(\sqrt{4\omega-v^2}x)- \frac{v}{2\sqrt{\omega}}\right\}\right]^{-1/2}.
\end{equation}
The parameters $(\omega,v)\in \Real^2$ describe the speed and the width of the soliton and  are connected to the pole $z_1$ by
\begin{equation}\label{e omega v}
        \omega=4|z_1|^2,\qquad
         v=-4\re(z_1).
\end{equation}
Note that $v^2<4\omega$ is automatically fulfilled if $z_1\in\Compl_+$. The norming constant $c_1$ influences only the phase and the spatial position of the soliton. To be precise we have
\begin{equation}\label{e x_0 gamma}
    x_0=2\ln\left[\frac{|c_1|}{2\im(z_1)}\right] \left(\sqrt{4\omega-v^2}\right)^{-1},
    \quad
    \gamma=\arg(c_1)+\frac{\pi}{2}+\frac{1}{2}\arg(z_1).
\end{equation}
Expressions for $N$-solitons with $N\geq2$ are large and not presented here. If $\re(z_j)
\neq\re(z_k)$ for $j\neq k$, then for large $|t|$, $N$-solitons break up into $N$  individual solitons of the form (\ref{e soliton}):
\begin{equation}\label{e sol sep}
    u^{(N\text{-sol})}(x,t)\sim \sum_{k=1}^N u_{\omega_k,v_k,x_{0,k}^{\pm},\gamma_k^{\pm}}(x,t), \quad\text{as }t\to\pm\infty.
\end{equation}
If the real parts of two poles $z_j$ and $z_k$ coincide, we obtain a solution having two peaks traveling at the same speed and the separation (\ref{e sol sep}) will not occur. Instead, \emph{breather} phenomena will appear. 
\section{Inverse scattering without poles}\label{s inverse sc}
In this section we are dealing with \rh \ref{rhp m} in the case where $N=0$. Hence, $m$ has no pole in $\Compl\setminus\Real$ and is analytic in $\Compl\setminus\Real$. We recall the associated Riemann--Hilbert problem:\\
\begin{framed}
    \begin{rhp}\label{rhp m^0}
        Find for each $x\in\Real$ a $2\times 2$-matrix valued function $\Compl\ni z\mapsto m(z;x)$ which satisfies
        \begin{enumerate}[(i)]
          \item $m(z;x)$ is meromorphic in $\Compl\setminus\Real$ (with respect to the parameter $z$).
          \item $m(z;x)=1+\mathcal{O}\left(\frac{1}{z}\right)$ as $|z|\to\infty$.
          \item The non-tangential boundary values $m_{\pm}(z;x)$ exist for $z\in\Real$ and satisfy the jump relation
              \begin{equation}\label{e jump m^0}
                  m_+=m_-(1+R),\quad\text{where}\quad
                  R(z;x):=
                  \left(
                    \begin{array}{cc}
                       \overline{r}_+(z)r_-(z) & e^{-2\ii zx}\overline{r}_+(z) \\
                       e^{2\ii zx}r_-(z) & 0 \\
                    \end{array}
                  \right).
              \end{equation}
        \end{enumerate}
    \end{rhp}
\end{framed}
For any function $h\in L^p(\Real)$ with $1\leq p<\infty$, the Cauchy operator denoted by $\pazocal{C}$ is given by
\begin{equation*}
    \pazocal{C}(h)(z):=\frac{1}{2\pi\ii}\int_{\Real} \frac{h(s)}{s-z}ds,\quad z\in\Compl\setminus\Real.
\end{equation*}
When $z$ approaches to a point on the real line transversely from the upper and lower half planes, the Cauchy operator becomes the following projection operators:
\begin{equation*}
    \Ppm(h)(z):=\lim_{\eps\downarrow 0}\frac{1}{2\pi\ii}\int_{\Real} \frac{h(s)}{s-(z\pm\eps)}ds,\quad z\in\Real.
\end{equation*}
The following proposition summarizes all properties which are needed to establish the solvability of \rh \ref{rhp m^0} and furthermore to prove estimates on the solution.
\begin{prop}\label{p cauchy operator}
    \begin{enumerate}[(i)]
      \item For every $h\in L^p(\Real)$, $1\leq p<\infty$, the Cauchy operator $\pazocal{C}(h)$ is analytic off the real line.
      \item For $h\in L^1(\Real)$, $\pazocal{C}(h)(z)$ decays to zero as $|z|\to\infty$ and admits the asymptotic
          \begin{equation}\label{e lim z Ch(z)}
            \lim_{|z|\to\infty}z\pazocal{C}(h)(z)= -\frac{1}{2\pi\ii}\int_\Real h(s) ds,
          \end{equation}
          where the limit is taken either in $\Compl^+$ or $\Compl^-$.
      \item The projection operators $\Ppm$ are linear bounded operators $L^p(\Real)\to L^p(\Real)$ for each $p\in(1,\infty)$. For $p=2$ we have $\|\Ppm\|_{L^2\to L^2}=1$.
      \item For every $x_0\in\Real_+$ and every $r\in H^1(\Real)$, we have
          \begin{equation}\label{e sup P^pm r 1}
              \sup_{x\in(x_0,\infty)} \|\langle x\rangle\Ppm (r(z)e^{\mp2\ii zx})\|_{L^{2}_z(\Real)}\leq \|r\|_{H^1},
          \end{equation}
          where $\langle x\rangle:=\sqrt{1+|x|^2}$. In addition,
          \begin{equation}\label{e sup P^pm r 2}
              \sup_{x\in\Real} \|\Ppm (r(z)e^{\mp2\ii zx})\|_{L^{\infty}_z(\Real)}\leq \frac{1}{\sqrt{2}}\|r\|_{H^1}.
          \end{equation}
          Furthermore, if $r\in L^{2,1}(\Real)$, then
          \begin{equation}\label{e sup P^pm r 3}
              \sup_{x\in\Real}\|\Ppm (zr(z)e^{\mp2\ii zx})\|_{L^{\infty}_z(\Real)}\leq \frac{1}{\sqrt{2}}\|zr\|_{L^{2,1}}.
          \end{equation}
      \item (Sokhotski-Plemelj theorem) The following two identities hold:
      \begin{equation}\label{e Sokhotski-Plemelj}
        \begin{aligned}
            &\Pp-\Pm=\text{Id}_{L^p(\Real)},\\
            &\Pp+\Pm=-\ii\pazocal{H},
        \end{aligned}
      \end{equation}
      where $\pazocal{H}:L^p(\Real)\to L^p(\Real)$ is the Hilbert transform given by
      \begin{equation*}
        \pazocal{H}(h)(z):=\lim_{\eps\downarrow 0}\frac{1}{\pi}\left(\int_{-\infty}^{z-\eps}+ \int^{\infty}_{z+\eps}\right) \frac{h(s)}{s-z}ds,\quad z\in\Real.
      \end{equation*}
      \item Let $f_+$ and $f_-$ functions defined in the upper (lower) $\Compl$-plane. If $f_{\pm}$ is analytic in $\Compl^{\pm}$ and $f_{\pm}(z)\to 0$ as $|z|\to\infty$ for $\im(z)\gtrless0$, then
          \begin{equation}\label{e Ppm of analytic functions}
            \Ppm(f_{\mp})(z)=0,\qquad\Ppm(f_{\pm})(z)=\pm f_{\pm}(z),\quad z\in\Real.
          \end{equation}
    \end{enumerate}
\end{prop}
The Cauchy operator is useful to convert \rh \ref{rhp m^0} into an integral equation. Indeed, the jump condition (\ref{e jump m^0}) can be written as
\begin{equation*}
    (m_+(z;x)-1)-(m_-(z;x)-1)=m_-(z;x)R(z;x).
\end{equation*}
Applying $\Pp$ and $\Pm$ to this equation yields by (\ref{e Ppm of analytic functions}) the following integral equation
\begin{equation}\label{e integral equation for m+-}
    m_{\pm}(z;x)=1+\Ppm(m_-(\cdot;x)R(\cdot;x))(z),\quad z\in\Real,
\end{equation}
which represents the solution of \rh \ref{rhp m^0} on the real line. The following Lemma ensures the solvability of \rh \ref{rhp m} (see Corollary 6 and Lemma 9 in \cite{Pelinovsky2016}):
\begin{lem}\label{l solvability of RHP N=0}
    Let $r_{\pm}\in H^1(\Real)\cap L^{2,1}(\Real)$ such that the relation (\ref{e relation r+ r-}) and the constraint (\ref{e r constraint}) hold. Then there exists an unique solution $m_{\pm}$ of the system of integral equations (\ref{e integral equation for m+-}). Moreover there exists a positive constant $C$ that depends on $\|r_{\pm}\|_{L^{\infty}}$ only such that $m_{\pm}$ enjoys the estimate
    \begin{equation*}
        \|m_{\pm}(\cdot;x)-1\|_{L^2}\leq C(\|r_{+}\|_{L^2} +\|r_{-}\|_{L^2})
    \end{equation*}
    for every $x\in\Real$.
\end{lem}
This Lemma yields indeed a solution of \rh \ref{rhp m^0}, since the analytic continuation of $m_{\pm}$ is found by Proposition \ref{p cauchy operator} (ii):
\begin{equation}\label{e RHP solution formula}
    m(z;x)=1+\frac{1}{2\pi\ii}\int_{\Real} \frac{m_-(y;x)R(y;x)}{y-z}dy,\quad z\in\Compl\setminus \Real.
\end{equation}
Alternatively we can factorize $1+R=(1+R_+)(1+R_-)$ with
\begin{equation}\label{e def R+ and R-}
    R_+(z;x)=
    \left(
      \begin{array}{cc}
            0 & e^{\overline{\phi}(z)}\overline{r}_+(z) \\
            0 & 0 \\
      \end{array}
    \right),\qquad
    R_-(z;x)=
    \left(
      \begin{array}{cc}
            0 & 0 \\
            e^{\phi(z)}r_-(z) & 0 \\
      \end{array}
    \right).
\end{equation}
The jump relation (\ref{e jump}) then becomes $m_+ -m_-=m_-R_+ +m_+R_-$ and applying again $\Ppm$ to this equation yields us
\begin{equation}\label{e alternative RHP solution formula}
    m(z;x)=1+\frac{1}{2\pi\ii}\int_{\Real} \frac{m_-(y;x)R_+(y;x)+ m_+(y;x)R_-(y;x)}{y-z}dy.
\end{equation}
In component form, for the non-tangential limits $z\to\Real$, we find
\begin{equation}\label{e component RHP solution formula}
    m_{\pm}(z;x)=1+
    \left[
      \begin{array}{cc}
        \Ppm\left([m_+(z;x)]_{12} r_-(z) e^{2\ii zx}\right)(z) & \Ppm\left([m_-(z;x)]_{11} \overline{r}_+(z) e^{-2\ii zx}\right)(z) \\
        \Ppm\left([m_+(z;x)]_{22} r_-(z) e^{2\ii zx}\right)(z) & \Ppm\left([m_-(z;x)]_{21} \overline{r}_+(z) e^{-2\ii zx}\right)(z) \\
      \end{array}
    \right].
\end{equation}
In the further analysis of \rh \ref{rhp m} we will meet expressions of the form
\begin{equation}\label{e def I_1,2}
    \begin{aligned}
        &I_1(r)(x):=\frac{1}{2\pi\ii} \int_{\Real}[m_-(y;x)-1]_{11} r(y) e^{-2\ii yx}dy,\\
        &I_2(r)(x):=\frac{1}{2\pi\ii} \int_{\Real}[m_+(y;x)-1]_{22}r(y) e^{2\ii yx}dy,
    \end{aligned}
\end{equation}
where $m_{\pm}$ are the unique solutions of the system of integral equations (\ref{e integral equation for m+-}) and $r$ is some given function.
\begin{prop}\label{p bound <x>^2I}
    Suppose that the assumptions of Lemma \ref{l solvability of RHP N=0} are fulfilled and take $r\in H^1(\Real)\cap L^{2,1}(\Real)$. Then the functionals defined in (\ref{e def I_1,2}) satisfy the bound
    \begin{equation}\label{e bound <x>^2I}
        \begin{aligned}
            &\|I_1(r)\|_{H^1(\Real_+)\cap L^{2,1}(\Real_+)}\leq C\|r_-\|_{H^1\cap L^{2,1}}(\|r_+\|_{H^1\cap L^{2,1}}+\|r_-\|_{H^1\cap L^{2,1}})\|r\|_{H^1\cap L^{2,1}},\\
            &\|I_2(r)\|_{H^1(\Real_+)\cap L^{2,1}(\Real_+)}\leq C\|r_+\|_{H^1\cap L^{2,1}}(\|r_+\|_{H^1\cap L^{2,1}}+\|r_-\|_{H^1\cap L^{2,1}})\|r\|_{H^1\cap L^{2,1}}
        \end{aligned}
    \end{equation}
    where $C$ is a positive constant.
\end{prop}
\begin{proof}
    For the convenience of the reader we prove this proposition although it is already proven in \cite{Pelinovsky2016}. We find by (\ref{e component RHP solution formula}) and integrating by parts
    \begin{eqnarray*}
      I_1(r)(x) &=& \frac{1}{2\pi\ii} \int_{\Real}\Pm\left([m_+(z;x)]_{12} r_-(z) e^{2\ii zx}\right)\!(y)\: r(y) e^{-2\ii yx}dy\\
       &=&   \frac{-1}{2\pi\ii} \int_{\Real}[m_+(y;x)]_{12} r_-(y) e^{2\ii zx}\Pp\left(r(z) e^{-2\ii zx}\right)(y) dy.
    \end{eqnarray*}
    Using the H\"{o}lder inequality and the estimate (\ref{e sup P^pm r 1}), we arrive at
    \begin{equation*}
        \sup_{x\in(x_0,\infty)}|\langle x\rangle^2I_1(r)(x)|\leq \|r_-\|_{L^{\infty}}\|r\|_{H^1}\sup_{x\in(x_0,\infty)} \|\langle x\rangle[m_+(y;x)]_{12}\|_{L^2_z}.
    \end{equation*}
    We know $\sup_{x\in(x_0,\infty)} \|\langle x\rangle[m_+(y;x)]_{12}\|_{L^2_z}\leq C\|r_+\|_{H^1}$ by \cite[Lemma 10]{Pelinovsky2016} which completes the proof of $I_1(r)\in L^{2,1}(\Real_+)$. The assertion $\partial_x I_1(r)\in L^{2}(\Real_+)$ is established by using again the inhomogeneous equation (\ref{e component RHP solution formula}), its $x$ derivative, integration by parts, H\"{o}lder inequality, and in the end estimates (\ref{e sup P^pm r 1}) - (\ref{e sup P^pm r 3}),
    \begin{equation*}
        \sup_{x\in(x_0,\infty)} \|\langle x\rangle[m_+(y;x)]_{12}\|_{L^2_z}\leq C\|r_+\|_{H^1}
    \end{equation*}
    and
    \begin{equation*}
        \sup_{x\in\Real} \|[\partial_x m_+(y;x)]_{12}\|_{L^2_z}\leq C(\|r_+\|_{H^1\cap L^{2,1}}+\|r_-\|_{H^1\cap L^{2,1}}).
    \end{equation*}
    The latter statement can also be found in \cite[Lemma 10]{Pelinovsky2016}.
\end{proof}
The proposition above yields directly the following fundamental result (see Lemma 11 in \cite{Pelinovsky2016}):
\begin{cor}\label{c u in H^11 and H^2 (positive half line)}
   Fix $M>0$. Under the assumptions of Lemma \ref{l solvability of RHP N=0} and if $\|r_+\|_{H^1\cap L^{2,1}}+\|r_-\|_{H^1\cap L^{2,1}}\leq M$, the potential $u$ reconstructed from the solution $m$ of \rh \ref{rhp m^0} by using (\ref{e rec 1}) and (\ref{e rec 2}) lies in $H^2(\Real_+)\cap H^{1,1}(\Real_+)$. Moreover, it satisfies the bound
   \begin{equation}\label{e u in H^11 and H^2 (positive half line)}
       \|u\|_{H^2(\Real_+)\cap H^{1,1}(\Real_+)}\leq C_M,
   \end{equation}
   where the constant $C_M$ does not depend on $r_{\pm}$.
\end{cor}
\begin{proof}
    We set
    \begin{equation}\label{e def w}
        w(x):=u(x)e^{\ii\int_{+\infty}^x|u(y)|^2dy}
    \end{equation}
    and
    \begin{equation}\label{e def v}
        v(x):=\overline{u}(x)e^{-\frac{1}{2\ii} \int_{+\infty}^x|u(y)|^2dy},
    \end{equation}
    such that the following relations hold:
    \begin{equation}\label{e relations u v w}
        \begin{aligned}
            |u(x)|&=|v(x)|=|w(x)|\\
            |u_x(x)|&\leq |v_x(x)|+\frac{1}{2}|v(x)|^3
        \end{aligned}
    \end{equation}
    Using the reconstruction formulas (\ref{e rec 1}) and (\ref{e rec 2}), Proposition \ref{p cauchy operator} (ii) and the integral equation (\ref{e alternative RHP solution formula}) we immediately find
    \begin{equation*}
        w(x)=\frac{2}{\pi\ii}\int_{\Real}\overline{r}_+(z) e^{-2\ii zx}dz\;+\;4\,I_1(\overline{r}_+)(x)
    \end{equation*}
    and
    \begin{equation*}
        e^{-\frac{1}{2\ii} \int_{+\infty}^x|u(y)|^2dy}v_x(x)
        =-\frac{1}{\pi}\int_{\Real}r_-(z) e^{2\ii zx}dz\;-\;2\ii\:I_2(r_-)(x).
    \end{equation*}
    In each of these equations the first summand on the right hand side is controlled in $H^1\cap L^{2,1}$ since $r_{\pm}\in H^1\cap L^{2,1}$. Moreover Proposition \ref{p bound <x>^2I} yields directly $w\in L^{2,1}(\Real_+)$ and $v_x\in L^{2,1}(\Real_+)$ and thus finally by (\ref{e relations u v w}) $u\in H^{1,1}(\Real_+)$. Proposition \ref{p bound <x>^2I} also leads to
    \begin{equation*}
        \partial_x\left(e^{-\frac{1}{2\ii} \int_{+\infty}^x|u(y)|^2dy}v_x(x)\right)\in L^2_x(\Real_+).
    \end{equation*}
    By a straightforward calculation we conclude $u\in H^2(\Real_+)$. The bound (\ref{e u in H^11 and H^2 (positive half line)}) is obtained from application of (\ref{e bound <x>^2I}). The proof of the Corollary is now complete.
\end{proof}
With regard to the B\"{a}klund transformation which we intend to use in the following section in order to include solitons we need the following Lemma in addition to (\ref{e u in H^11 and H^2 (positive half line)}). The only purpose in the repeating of so many details of the inverse Scattering withour poles is to deduce this Lemma which can not be found in \cite{Pelinovsky2016}.
\begin{lem}\label{l m(z0)-1 in H^11 and H^2}
   Let the assumptions of Corollary \ref{c u in H^11 and H^2 (positive half line)} be valid and fix $z_0\in\Compl\setminus\Real$. Then for the solution $m(z;x)$ of \rh \ref{rhp m^0} we have $m(z_0;\cdot)-1\in H^1(\Real_+)\cap L^{2,1}(\Real_+)$ with the bound
   \begin{equation}\label{e m(z0)-1 in H^1 and L^2,1}
       \|m(z_0;\cdot)-1\|_{H^1(\Real_+)\cap L^{2,1}(\Real_+)}\leq C_{M},
   \end{equation}
   where the constant $C_M$ depends on $z_0$ and $M$ but not on $r_{\pm}$.
\end{lem}
\begin{proof}
    Fix $z_0\in\Compl\setminus\Real$. We use (\ref{e alternative RHP solution formula}) to find
    \begin{equation}\label{e m21 m12 in L^2,1}
        \begin{aligned}
            &[m(z_0;x)]_{12}= \frac{1}{2\pi\ii} \int_{\Real}\frac{[m_-(y;x)]_{11}\overline{r}_+(y) e^{-2\ii yx}}{y-z_0}dy= \frac{1}{2\pi\ii}\int_{\Real}\widetilde{r}_+(z) e^{-2\ii zx}dz\;+\;I_1(\widetilde{r}_+)(x),\\
            &[m(z_0;x)]_{21}= \frac{1}{2\pi\ii} \int_{\Real}\frac{[m_+(y;x)]_{22}r_-(y) e^{2\ii yx}}{y-z_0}dy= \frac{1}{2\pi\ii}\int_{\Real}\widetilde{r}_-(z) e^{2\ii zx}dz\;+\;I_2(\widetilde{r}_-)(x),
        \end{aligned}
    \end{equation}
    where $\widetilde{r}_-(z):={r}_-(z)/(z-z_0)$ and $\widetilde{r}_+(z):=\overline{r}_+(z)/(z-z_0)$, respectively. Due to the fact that $\|\widetilde{r}_{\pm}\|_{H^1\cap L^{2,1}}\leq c \|{r}_{\pm}\|_{H^1\cap L^{2,1}}$, where the constant $c>0$ depends on $z_0$ only, and using Proposition \ref{p bound <x>^2I} we end up with (\ref{e m(z0)-1 in H^1 and L^2,1}) for the non diagonal entries $m_{12}$ and $m_{21}$.
    Using again (\ref{e alternative RHP solution formula}) we obtain
    \begin{equation*}
        [m(z_0;x)]_{11}=1+\frac{1}{2\pi\ii} \int_{\Real} \frac{[m_+(y;x)]_{12}r_-(y) e^{2\ii yx}}{y-z_0}dy,
    \end{equation*}
    where we can insert $[m_+(y;x)]_{12}=\Pp ([m_-(z;x)]_{11} \overline{r}_+(z) e^{-2\ii zx})(y)$ from the integral equation (\ref{e integral equation for m+-}). Then we integrate by parts and obtain
    \begin{equation}\label{e formula for m(z_0;x)-1 in the proof of the lemma}
        [m(z_0;x)]_{11}=1-\frac{1}{2\pi\ii} \int_{\Real} [m_-(y;x)]_{11} \overline{r}_+(y) e^{-2\ii yx}\;\Pm (\widetilde{r}_-(z) e^{2\ii zx})(y)\:dy,
    \end{equation}
    where we put again $\widetilde{r}_-(z):={r}_-(z)/(z-z_0)$. Furthermore we set
    \begin{equation*}
        R_+(y):=\overline{r}_+(y)\;\Pm (\widetilde{r}_-(z) e^{2\ii zx})(y).
    \end{equation*}
    To prove $R_+\in H^1\cap L^{2,1}$ we recall the continuity property $\|\Ppm\|_{L^2\to L^2}=1$. One consequence is that $\|\Pm (\widetilde{r}_-(z) e^{2\ii zx})(\cdot)\|_{L^2} \leq c \|r_-\|_{L^2}$. Additionally, we find
    \begin{equation*}
        \|\partial_z \Pm (\widetilde{r}_-(z) e^{2\ii zx})(z)\|_{L^2_z}\leq \| \Pm (\widetilde{r}'_-(z) e^{2\ii zx})(\cdot)\|_{L^2_z} +\|2\ii x \Pm (\widetilde{r}_-(z) e^{2\ii zx})(\cdot)\|_{L^2_z},
    \end{equation*}
    where we can apply the bound (\ref{e sup P^pm r 1}) of Proposition \ref{p cauchy operator} and again $\|\Ppm\|_{L^2\to L^2}=1$. Thus, we are able to control  $\Pm (\widetilde{r}_-(z) e^{2\ii zx})(\cdot)$ in $H^1$ uniformly for $x>0$. Altogether, we have shown $\|R_+\|_ {H^1\cap L^{2,1}}\leq c \|r_+\|_ {H^1\cap L^{2,1}}\|r_-\|_ {H^1}$, which is needed because we want to apply Proposition \ref{p bound <x>^2I}. Therefore we write (\ref{e formula for m(z_0;x)-1 in the proof of the lemma}) in the form
    \begin{equation*}
        [m(z_0;x)]_{11}-1=\frac{1}{2\pi\ii} \int_{\Real} R_+(y)e^{-2\ii yx}\:dy+ \:I_1(R_+\!)(x).
    \end{equation*}
    Analogously, it can be carried out in a similar way, that for $R_-(y):=r_-(y)\;\Pp (\widetilde{r}_+(z) e^{-2\ii zx})(y)$,
    \begin{equation*}
        [m(z_0;x)]_{22}-1=\frac{1}{2\pi\ii} \int_{\Real} R_-(y)e^{2\ii yx}\:dy+ \:I_2(R_-\!)(x).
    \end{equation*}
    Combining Fourier theory and the bound (\ref{e bound <x>^2I}) we have now accomplished the proof of (\ref{e m(z0)-1 in H^1 and L^2,1}) also for the diagonal entries.
\end{proof}
Estimates on the negative half-line can be found by modifying the solution $m(z;x)$ of \rh \ref{rhp m^0} in the following way:
\begin{equation}\label{e def m delta}
    m_{\delta}(z;x):=m(z;x)
    \left[
      \begin{array}{cc}
        \delta^{-1}(z) & 0 \\
        0 & \delta(z) \\
      \end{array}
    \right],
\end{equation}
where
\begin{equation}\label{e delta}
    \delta(z)=\exp\left(\frac{1}{2\pi\ii} \int_{\Real}\frac{\log(1+ \overline{r}_+(y)r_-(y))}{y-z}dy\right).
\end{equation}
In Proposition 8 in \cite{Pelinovsky2016} it is shown that $\log(1+ \overline{r}_+r_-)\in L^2(\Real)$ due to (\ref{e r constraint}). Hence, the integral in (\ref{e delta}) is well-defined and $\delta$ solves the following RHP:
\begin{framed}
    \begin{rhp}\label{rhp delta}
        Find a scalar valued function $\Compl\ni z\mapsto \delta(z)$ which satisfies
        \begin{enumerate}[(i)]
          \item $\delta(z)$ is meromorphic in $\Compl\setminus\Real$.
          \item $\delta(z)=1+\mathcal{O} \left(\frac{1}{z}\right)$ as $|z|\to\infty$.
          \item The non-tangential boundary values $\delta_{\pm}(z)$ exist for $z\in\Real$ and satisfy the jump relation
              \begin{equation}\label{e jump delta}
                 \delta_+(z)=\left[1+ \phantom{\widehat{l}}\overline{r}_+(z)r_-(z) \right] \delta_-(z).
              \end{equation}
        \end{enumerate}
    \end{rhp}
\end{framed}
Using the symmetry $\delta(\overline{z})=\overline{\delta}^{-1}(z)$ and the jump condition (\ref{e jump delta}) it is an easy exercise to verify, that the function $m_{\delta}(z;x)$ defined in \ref{e def m delta} is a solution to the following \rh:
\begin{samepage}
\begin{framed}
    \begin{rhp}\label{rhp m^0 delta}
        Find for each $x\in\Real$ a $2\times 2$-matrix valued function $\Compl\ni z\mapsto m_{\delta}(z;x)$ which satisfies
        \begin{enumerate}[(i)]
          \item $m_{\delta}(z;x)$ is meromorphic in $\Compl\setminus\Real$ (with respect to the parameter $z$).
          \item $m_{\delta}(z;x)=1+\mathcal{O} \left(\frac{1}{z}\right)$ as $|z|\to\infty$.
          \item The non-tangential boundary values $m_{\pm,\delta}(z;x)$ exist for $z\in\Real$ and satisfy the jump relation
              \begin{equation}\label{e jump m^0 delta}
                  m_{+,\delta}=m_{-,\delta}(1+R_{\delta}) ,\quad\text{where}\quad
                  R_{\delta}(z;x):=
                  \left[
                    \begin{array}{cc}
                       0 & e^{-2\ii zx}\overline{r}_{+,\delta}(z) \\
                       e^{2\ii zx}r_{-,\delta}(z) & \overline{r}_{+,\delta}(z) r_{-,\delta}(z) \\
                    \end{array}
                  \right],
              \end{equation}
              and $r_{\pm,\delta}(z):= \overline{\delta}_+(z) \overline{\delta}_-(z)r_{\pm}(z)$.
        \end{enumerate}
    \end{rhp}
\end{framed}
\end{samepage}
The new jump matrix $R_{\delta}$ admits an factorization analogously to (\ref{e def R+ and R-}). For
\begin{equation}\label{e def R+ and R- delta}
    R_{+,\delta}(z;x):=
                  \left[
                    \begin{array}{cc}
                       0 & 0 \\
                       e^{2\ii zx}r_{-,\delta}(z) & 0 \\
                    \end{array}
                  \right],\qquad
    R_{-,\delta}(z;x):=
                  \left[
                    \begin{array}{cc}
                       0 & e^{-2\ii zx}\overline{r}_{+,\delta}(z) \\
                       0 & 0 \\
                    \end{array}
                  \right],
\end{equation}
we find
$m_{+,\delta} -m_{-,\delta}=m_{-,\delta}R_{+,\delta} +m_{+,\delta}R_{-,\delta}$ for $z\in \Real$, such that analogously to (\ref{e alternative RHP solution formula}),
\begin{equation*}
    m_{\delta}(z;x)=1+\frac{1}{2\pi\ii}\int_{\Real} \frac{m_{-,\delta}(y;x)R_{+,\delta}(y;x)+ m_{+,\delta}(y;x)R_{-,\delta}(y;x)}{y-z}dy.
\end{equation*}
The following exemplary calculation shows why \rh \ref{rhp m^0 delta} can be studied in order to extend Lemma \ref{l m(z0)-1 in H^11 and H^2} and Corollary \ref{c u in H^11 and H^2 (positive half line)} to the negative half-line. We have for $z_0\in\Compl\setminus\Real$
\begin{eqnarray*}
  [m_{\delta}(z_0;x)]_{12}&=& \frac{1}{2\pi\ii} \int_{\Real}\frac{[m_{+,\delta}(y;x)]_{11} \overline{r}_{+,\delta}(y) e^{-2\ii yx}}{y-z_0}dy \\
   &=&  \frac{1}{2\pi\ii} \int_{\Real}\widetilde{r}_{+,\delta}(y) e^{-2\ii yx}dy+I_{1,\delta}(\widetilde{r}_{+,\delta})
\end{eqnarray*}
where $\widetilde{r}_{+,\delta}(z):= \overline{r}_{+,\delta}(z)/(z-z_0)$ and
\begin{equation*}
    I_{1,\delta}(r):=\frac{1}{2\pi\ii} \int_{\Real}[m_{+,\delta}(y;x)-1]_{11} r(y) e^{-2\ii yx}dy.
\end{equation*}
The functional $I_{1,\delta}(r)$ satisfies the same estimates as in Proposition \ref{p bound <x>^2I} with $\Real_+$ replaced by $\Real_-$ because the operators $\Pp$ and $\Pm$ swap their places in comparison with the integral equation (\ref{e integral equation for m+-}).
\begin{lem}\label{l m(z0)-1 in in H^1 and L^2,1 (delta)}
   Fix $M>0$ and $z_0\in\Compl\setminus\Real$ and let the assumptions of Lemma \ref{l solvability of RHP N=0} be valid. If in addition $\|r_+\|_{H^1\cap L^{2,1}}+\|r_-\|_{H^1\cap L^{2,1}}\leq M$, then for the solution $m_{\delta}(z;x)$ of \rh \ref{rhp delta} we have $m_{\delta}(z_0;\cdot)-1\in H^1(\Real_-)\cap L^{2,1}(\Real_-)$ with the bound
   \begin{equation*}
       \|m_{\delta}(z_0;\cdot)-1\|_{H^1(\Real_-)\cap L^{2,1}(\Real_-)}\leq C_{M},
   \end{equation*}
   where the constant $C_M$ depends on $z_0$ and $M$, but not on $r_{\pm}$.
\end{lem}
With respect to the potential $u(x)$ the two Riemann--Hilbert problems \ref{rhp m^0} and \ref{rhp m^0 delta} are equivalent in the following sense:
\begin{equation}\label{e equivalent rhps}
    \begin{aligned}
        \lim_{|z|\to\infty}z\;[m(z;x)]_{12}= \lim_{|z|\to\infty}z\;[m_{\delta}(z;x)]_{12},\\
        \lim_{|z|\to\infty}z\;[m(z;x)]_{21}= \lim_{|z|\to\infty}z\;[m_{\delta}(z;x)]_{21}.
    \end{aligned}
\end{equation}
This observation follows directly from the definition (\ref{e def m delta}) and leads to the following extension of Corollary \ref{c u in H^11 and H^2 (positive half line)}.
\begin{cor}\label{c u in H^11 and H^2 (negative half line)}
   Fix $M>0$. Under the assumptions of Lemma \ref{l solvability of RHP N=0} and if $\|r_+\|_{H^1\cap L^{2,1}}+\|r_-\|_{H^1\cap L^{2,1}}\leq M$, the potential $u$ reconstructed from the solution $m$ of \rh \ref{rhp m^0} by using (\ref{e rec 1}) and (\ref{e rec 2}) lies in $H^2(\Real_-)\cap H^{1,1}(\Real_-)$ and satisfies the bound
   \begin{equation}\label{e u in H^11 and H^2 (negative half line)}
       \|u\|_{H^2(\Real_-)\cap H^{1,1}(\Real_-)}\leq C_M
   \end{equation}
   where the constant $C_M$ does not depend on $r_{\pm}$.
\end{cor}

\section{Adding a pole}\label{s adding a pole}
In this section we want to prove the solvability of \rh \ref{rhp m} if $N=1$. An auto-B\"{a}cklund transformation will establish a connection between the cases $N=1$ and $N=0$. All formulas were found in \cite{Deift2011} and \cite{Cuccagna2014}, where the B\"{a}cklund transformation was used in the context of the NLS equation. \\
Assume that a function $u^{(1)}\in H^2\cap H^{1,1}$ provides scattering data $\mathcal{S}^{(1)}=\set{r^{(1)}_{\pm};z_1;c_1}$. We recall the corresponding Riemann--Hilbert problem (without time dependence):
\begin{samepage}
\begin{framed}
    \begin{rhp}\label{rhp m^1}
        Find for each $x\in\Real$ a $2\times 2$-matrix valued function $\Compl\ni z\mapsto m^{(1)}(z;x)$ which satisfies
        \begin{enumerate}[(i)]
          \item $m^{(1)}(z;x)$ is meromorphic in $\Compl\setminus\Real$ (with respect to the parameter $z$).
          \item $m^{(1)}(z;x)=1+\mathcal{O}\left(\frac{1}{z}\right)$ as $|z|\to\infty$.
          \item The non-tangential boundary values $m^{(1)}_{\pm}(z;x)$ exist for $z\in\Real$ and satisfy the jump relation
              \begin{equation}\label{e jump m^1}
                  m^{(1)}_+=m^{(1)}_-(1+R), \quad\text{where}\quad
                  R(z;x):=
                  \left[
                    \begin{array}{cc}
                       \overline{r}_+(z)r_-(z) & e^{-2\ii zx}\overline{r}_+(z) \\
                       e^{2\ii zx}r_-(z) & 0 \\
                    \end{array}
                  \right].
              \end{equation}
          \item $m^{(1)}$ has simple poles at $z_1$ and $\overline{z}_1$ with
              \begin{equation}\label{e Res m^1}
                  \begin{aligned}
                     \res_{z=z_1}m^{(1)}(z;x)&=\lim_{z\to z_1}m^{(1)}(z;x)
                  \left[
                      \begin{array}{cc}
                        0 & 0 \\
                        2\ii\lambda_1c_1 e^{2\ii z_1x} & 0
                      \end{array}
                  \right],\\
                     \res_{z=\overline{z}_1} m^{(1)}(z;x)&=\lim_{z\to \overline{z}_1}m^{(1)}(z;x)
                  \left[
                      \begin{array}{cc}
                      0 & \frac{-\overline{c}_1 e^{-2\ii \overline{z}_1x}}{2\ii\lambda_1} \\
                      0 & 0
                      \end{array}
                  \right].
                  \end{aligned}
              \end{equation}
        \end{enumerate}
    \end{rhp}
\end{framed}
\end{samepage}
 By construction, the constraints (\ref{e relation r+ r-}) - (\ref{e r constraint}) hold. Now we change these data by removing the pole $z_1$ and modifying the reflection coefficient in the following way:
\begin{equation}\label{e r^0}
    r^{(0)}_{\pm}(z):=r^{(1)}_{\pm}(z) \frac{z-\overline{z}_1}{z-z_1}.
\end{equation}
Obviously, $r^{(0)}_{\pm}$ satisfy (\ref{e relation r+ r-}) -- (\ref{e r constraint}) and moreover, $r^{(1)}_{\pm}\in H^1\cap L^{2,1}$ implies $r^{(0)}_{\pm}\in H^1\cap L^{2,1}$. Hence, all assumptions of Lemma \ref{l solvability of RHP N=0} are satisfied and we get an unique solution $m^{(0)}(z;x)$ of \rh \ref{rhp m^0} with our new data $\mathcal{S}^{(0)}:=\set{r^{(0)}_{\pm}}$. This procedure defines a map $u^{(1)}(x)\mapsto u^{(0)}(x)$, where $u^{(0)}(x)$ is defined to be the pure radiation potential  which is associated to $m^{(0)}(z;x)$ by the reconstruction formulas (\ref{e rec 1}) and (\ref{e rec 1}), respectively.
\subsection{B\"{a}cklund transformation for $x>0$}
What we will do in this subsection is to explore the map $u^{(1)}\leftrightarrow u^{(0)}$ for $x>0$. Therefore we introduce the functions $w^{(j)}$, $v^{(j)}$ for $j=0,1$, which are related to $u^{(j)}$ by (\ref{e def w}) and (\ref{e def v}), respectively. Next we define the matrix
\begin{equation*}
    A(x)=
    \left[
      \begin{array}{cc}
        a_{11}(x) & a_{12}(x) \\
        a_{21}(x) & a_{22}(x) \\
      \end{array}
    \right]
\end{equation*}
by
\begin{equation*}
        \left(
          \begin{array}{c}
            a_{11}(x) \\
            a_{21}(x) \\
          \end{array}
        \right)
        :=m^{(0)}(z_1;x)
        \left(
          \begin{array}{c}
            1 \\
            - \frac{2\ii\lambda_1c_1 e^{2\ii z_1x}} {z_1-\overline{z}_1} \\
          \end{array}
        \right),
        \quad
        \left(
          \begin{array}{c}
            a_{12}(x) \\
            a_{22}(x) \\
          \end{array}
        \right)
        :=m^{(0)}(\overline{z}_1;x)
        \left(
          \begin{array}{c}
            \frac{\overline{c}_1 e^{-2\ii \overline{z}_1 x}} {2\ii\overline{\lambda}_1(\overline{z}_1-z_1)} \\
            1 \\
          \end{array}
        \right).
\end{equation*}
In order to define the B\"{a}cklund transformation it is necessary to know that there is no $x$ such that the determinant of $A(x)$ vanishes.
\begin{prop}\label{p A inverse}
    The matrix $A$ is invertible for all $x\in\Real$. Moreover, if $\|r^{(1)}_+\|_{H^1\cap L^{2,1}}+\|r^{(1)}_-\|_{H^1\cap L^{2,1}}\leq M$, then
    \begin{equation} \label{e lower bound for det A}
    |\det(A(x))|^{-1}\leq C_M,\quad\text{for all }x>0,
    \end{equation}
    where the constant $C_M$ does not depend on $x$ and $r_{\pm}$.
\end{prop}
\begin{proof}
    Using the symmetry (\ref{e symmetrie of m}) we find
    \begin{eqnarray*}
       \left(
          \begin{array}{c}
            a_{12}(x) \\
            a_{22}(x) \\
          \end{array}
       \right)
       &=&
       \left[
         \begin{array}{cc}
           w^{(0)}(x) & 1 \\
           -|w^{(0)}(x)|^2-4\overline{z}_1 & -\overline{w^{(0)}}(x) \\
         \end{array}
       \right]
       \overline{m^{(0)}}(z_1;x)
       \left[
          \begin{array}{cc}
            0 & \frac{-1}{4\overline{z}_1} \\
            1 & 0 \\
          \end{array}
       \right]
       \left(
          \begin{array}{c}
            \frac{\overline{c}_1 e^{-2\ii \overline{z}_1 x}} {2\ii\overline{\lambda}_1(\overline{z}_1-z_1)} \\
            1 \\
          \end{array}
       \right)
       \\
       &=&
       \frac{1}{4\overline{z}_1}
       \left[
         \begin{array}{cc}
           -w^{(0)}(x) & -1 \\
           |w^{(0)}(x)|^2+4\overline{z}_1 & \overline{w^{(0)}}(x) \\
         \end{array}
       \right]
       \overline{m^{(0)}}(z_1;x)
       \left(
          \begin{array}{c}
            1\\
            \frac{2\ii\overline{\lambda}_1\overline{c}_1 e^{-2\ii \overline{z}_1 x}} {(\overline{z}_1-z_1)} \\
          \end{array}
       \right)
       \\
       &=&
       \frac{1}{4\overline{z}_1}
       \left[
         \begin{array}{cc}
           -w^{(0)}(x) & -1 \\
           |w^{(0)}(x)|^2+4\overline{z}_1 & \overline{w^{(0)}}(x) \\
         \end{array}
       \right]
       \left(
          \begin{array}{c}
            \overline{a_{11}}(x) \\
            \overline{a_{21}}(x) \\
          \end{array}
       \right).
    \end{eqnarray*}
    It follows directly that
    \begin{equation*}
        \det(A(x))=|a_{11}(x)|^2 +\frac{1}{4\overline{z}_1} |\overline{w^{(0)}}(x)a_{11}(x)+a_{21}(x) |^2.
    \end{equation*}
    The case $\det(A(x))=0$ is impossible, since due to $\im(z_1)\neq 0$ it would follow that $a_{11}(x)=a_{21}(x)=0$ and hence $\left(1,
    - \frac{2\ii\lambda_1c_1 e^{2\ii z_1x}}{z_1-\overline{z}_1}\right)^T \in\ker[m^{(0)}]$. This contradicts $\det(m^{(0)}(z;x))\equiv 1$ (see Remark \ref{r uniqueness + det=1}). Now we turn to the proof of (\ref{e lower bound for det A}). For sake of contradiction we assume that for any $d>0$ we can find $x>0$ such that $|\det(A(x))|<d$. Due to $\im(z_1)\neq 0$ and $w\in L^{\infty}$ wa can assume w.l.o.g. $|a_{11}(x)|< d$ and $|a_{12}(x)|<d$. Using (\ref{e det m=1}) we find
    \begin{eqnarray*}
      1 &=& \left|[m^{(0)}(z_1;x)]_{11} [m^{(0)}(z_1;x)]_{22}- [m^{(0)}(z_1;x)]_{12} [m^{(0)}(z_1;x)]_{21}\right| \\
       &=& \left|\left\{a_{11}(x) +\frac{2\ii\lambda_1c_1 e^{2\ii z_1x}} {z_1-\overline{z}_1}[m^{(0)}(z_1;x)]_{12}\right\} [m^{(0)}(z_1;x)]_{22}\right.\\
         &&\,
         \left.-\,[m^{(0)}(z_1;x)]_{12} \left\{ a_{21}(x)+ \frac{2\ii\lambda_1c_1 e^{2\ii z_1x}} {z_1-\overline{z}_1}[m^{(0)}(z_1;x)]_{22}\right\} \right|\\
       &=& \left|a_{11}(x) [m^{(0)}(z_1;x)]_{22}- [m^{(0)}(z_1;x)]_{12} a_{21}(x)\right|\\
       &<&d\cdot   \left\{\left|[m^{(0)}(z_1;x)]_{22}\right|+ \left|[m^{(0)}(z_1;x)]_{12}\right|\right\}\\
       &\leq&d\cdot C\cdot  \|m^{(0)}(z_1;\cdot)-1\|_{H^1(\Real_+)\cap L^{2,1}(\Real_+)}\\
       &\leq& d\cdot C_M.
    \end{eqnarray*}
    Here $C_M$ is the constant in Lemma \ref{l m(z0)-1 in H^11 and H^2} and it follows that $d$ cannot be arbitrary small. In addition we also proved the bound (\ref{e lower bound for det A}).
\end{proof}
\begin{lem}\label{l solvability of RHP N=1}
    For any scattering data $\mathcal{S}^{(1)}=\{r^{(1)}_{\pm};z_1;c_1\}$ such that $r^{(1)}_{\pm}\in L^{2,1}\cap H^1$ satisfies (\ref{e relation r+ r-})-(\ref{e r constraint}), \rh \ref{rhp m^1} admits an unique solution $m^{(1)}(z;x)$. This solution can be obtained from $m^{(0)}(z;x)$ by the following:
    \begin{equation}\label{e Bäcklund for m^1}
        m^{(1)}(z;x)=A(x)\mu(z) A^{-1}(x)m^{(0)}(z;x)\mu^{-1}(z),
    \end{equation}
    where
    \begin{equation*}
        \mu(z)=
        \left[
          \begin{array}{cc}
            z-z_1 & 0 \\
            0 & z-\overline{z}_1 \\
          \end{array}
        \right].
    \end{equation*}
\end{lem}
\begin{proof}
    Let us denote by $\widetilde{m}(z;x)$ the right hand side of (\ref{e Bäcklund for m^1}) and set
    \begin{equation*}
        \left[
          \begin{array}{cc}
            \tau_{11}(z) & \tau_{12}(z) \\
            \tau_{21}(z) & \tau_{22}(z) \\
          \end{array}
        \right]:=A^{-1}(x)m^{(0)}(z;x).
    \end{equation*}
    We find
    \begin{equation}\label{e Res of m tilde}
        \begin{aligned}
          \res_{z=z_1}\widetilde{m}(z;x)&=
          A(x)
          \left[
            \begin{array}{cc}
              0 & 0 \\
              (z_1-\overline{z}_1)\tau_{21}(z_1) & 0 \\
            \end{array}
          \right]
          ,\\
          \res_{z=\overline{z}_1}\widetilde{m} (z;x)&=
          A(x)
          \left[
            \begin{array}{cc}
              0 & (\overline{z}_1-z_1) \tau_{12}(\overline{z}_1) \\
              0 & 0 \\
            \end{array}
          \right]
          ,
        \end{aligned}
    \end{equation}
    and
    \begin{equation}\label{e lim m tilde c}
        \begin{aligned}
        &\lim_{z\to z_1}\widetilde{m}(z;x)
          \left(
            \begin{array}{cc}
              0 & 0 \\
              2\ii\lambda_1c_1 e^{2\ii z_1x} & 0
            \end{array}
          \right)=
          A(x)
          \left[
            \begin{array}{cc}
              0 & 0 \\
              2\ii\lambda_1c_1 e^{2\ii z_1x}\tau_{22}(z_1) & 0 \\
            \end{array}
          \right]
          ,\\
          &\lim_{z\to \overline{z}_1}\widetilde{m}(z;x)
          \left(
            \begin{array}{cc}
              0 & \frac{-\overline{c}_1}{2\ii\lambda_1} e^{-2\ii \overline{z}_1x} \\
              0 & 0
            \end{array}
          \right)=
          A(x)
          \left[
            \begin{array}{cc}
              0 &\frac{-\overline{c}_1}{2\ii\lambda_1} \tau_{11}(\overline{z}_1) \\
              0 & 0 \\
            \end{array}
          \right].
        \end{aligned}
      \end{equation}
      Using $\det m^{(0)}\equiv 1 $ it is easy to obtain
      \begin{equation*}
        \tau_{21}(z_1)=\frac{1}{\det A(x)} \frac{2\ii\lambda_1c_1 e^{2\ii z_1x}}{z_1-\overline{z}_1},\quad
        \tau_{22}(z_1)=\frac{1}{\det A(x)},
      \end{equation*}
      and
      \begin{equation*}
        \tau_{11}(\overline{z}_1)=\frac{1}{\det A(x)},\quad\tau_{12}(\overline{z}_1)= \frac{-1}{\det A(x)} \frac{\overline{c}_1 e^{-2\ii \overline{z}_1x}}{2\ii\overline{\lambda}_1 (\overline{z}_1-z_1)},
      \end{equation*}
      and thus it follows from (\ref{e Res of m tilde}) and (\ref{e lim m tilde c}) that $\widetilde{m}$ satisfies (\ref{e Res m^1}). Now we proceed with the jump on the real axis and check if point (iii) of \rh \ref{rhp m^1} is satisfied. Using the jump condition of $m^{(0)}$ (see (\ref{e jump})) and the definition (\ref{e r^0}) of $r_{\pm}^{(0)}$ we find for $z\in\Real$
      \begin{eqnarray*}
        \widetilde{m}_+(z;x) &=& \widetilde{m}_-(z;x) \mu(z)\left(
          \begin{array}{cc}
            1+\overline{r}^{(0)}_+(z)r^{(0)}_-(z) & e^{-2\ii zx}\overline{r}^{(0)}_+(z) \\
            e^{2\ii zx}r^{(0)}_-(z) & 1 \\
          \end{array}
        \right)\mu^{-1}(z) \\
         &=&  \widetilde{m}_-(z;x)\left(
          \begin{array}{cc}
            1+\overline{r}^{(1)}_+(z)r^{(1)}_-(z) & e^{-2\ii zx}\overline{r}^{(1)}_+(z) \\
            e^{2\ii zx}r^{(1)}_-(z) & 1 \\
          \end{array}
        \right).
      \end{eqnarray*}
      Next we observe
      \begin{equation}\label{e expansion of m tilde}
        \widetilde{m}(z;x)= \left[1+\frac{A(x)\:\mu(0)\:A^{-1}(x)}{z}\right] m^{(0)}(z;x)
        \left[
          \begin{array}{cc}
            \frac{z}{z-z_1} & 0 \\
            0 & \frac{z}{z-\overline{z}_1} \\
          \end{array}
        \right].
      \end{equation}
      It follows that $\widetilde{m}$ behaves for $|z|\to\infty$ as required in the point (ii) of \rh \ref{rhp m^1}. Since also the point (i) of \rh \ref{rhp m^1} is true, we conclude by the uniqueness (see Remark \ref{r uniqueness + det=1}) that $m^{(1)}(z;x)\equiv\widetilde{m}(z;x)$.
\end{proof}
The B\"{a}cklund transformation formula (\ref{e Bäcklund for m^1}) is an ideal expression to extend Corollary \ref{c u in H^11 and H^2 (positive half line)} and Lemma \ref{l m(z0)-1 in H^11 and H^2} to the case where the scattering data are involving one pole $z_1$.
\begin{cor}\label{c u in H^11 and H^2 (positive half line - 1 pole)}
   Under the assumptions of Lemma \ref{l solvability of RHP N=1} the potential $u^{(1)}(x)$ reconstructed from the solution $m^{(1)}(z;x)$ of \rh \ref{rhp m dynamic} by using (\ref{e rec 1}) and (\ref{e rec 2}) lies in $H^2(\Real_+)\cap H^{1,1}(\Real_+)$. Moreover, if $\|r^{(1)}_+\|_{H^1\cap L^{2,1}}+\|r^{(1)}_-\|_{H^1\cap L^{2,1}}+|c_1|\leq M$ for some fixed $M>0$, then $u^{(1)}$ satisfies the bound
   \begin{equation}\label{e u^1 in H^11 and H^2 (positive half line)}
       \|u^{(1)}\|_{H^2(\Real_+)\cap H^{1,1}(\Real_+)}\leq C_M
   \end{equation}
   where the constant $C_M$ depends on $M$ and $z_1$ but not on $r^{(1)}_{\pm}$ and $|c_1|$.
\end{cor}
\begin{proof}
    We use (\ref{e expansion of m tilde}) and
    the expansion
    \begin{equation*}
        \left[
          \begin{array}{cc}
            \frac{z}{z-z_1} & 0 \\
            0 & \frac{z}{z-\overline{z}_1} \\
          \end{array}
        \right]=1-\frac{\mu(0)}{z}+\mathcal{O}(z^{-2}), \quad\text{as }|z|\to\infty
    \end{equation*}
    in order to find
    \begin{equation*}
        \begin{aligned}
            \lim_{|z|\to\infty}z\; \left[m^{(1)}(z;x)\right]_{12}& = \lim_{|z|\to\infty}z\; \left[m^{(0)}(z;x)\right]_{12}+ \left[A(x)\:\mu(0)\:A^{-1}(x)\right]_{12},\\
            \lim_{|z|\to\infty}z\; \left[m^{(1)}(z;x)\right]_{21}& = \lim_{|z|\to\infty}z\; \left[m^{(0)}(z;x)\right]_{21}+ \left[A(x)\:\mu(0)\:A^{-1}(x)\right]_{21}.
        \end{aligned}
    \end{equation*}
    Using the notation (\ref{e def w}) and (\ref{e def v}), we find by the reconstruction formulas (\ref{e rec 1}) and (\ref{e rec 2})
    \begin{equation}\label{e decomposition w}
       w^{(1)}(x)=w^{(0)}(x)+B_1(x),\qquad B_1(x):=-\frac{8\ii \im(z_1)a_{11}(x)a_{12}(x)}{\det(A(x))}
    \end{equation}
    and
    \begin{multline}\label{e decomposition v}
        \qquad e^{-\frac{1}{2\ii} \int_{+\infty}^x|u^{(1)}(y)|^2dy}v^{(1)}_x(x)= e^{-\frac{1}{2\ii} \int_{+\infty}^x|u^{(0)}(y)|^2dy}v^{(0)}_x(x) +B_2(x),\\
        B_2(x):=\frac{4 \im(z_1)a_{21}(x)a_{22}(x)}{\det(A(x))}.\qquad
    \end{multline}
    As it is easily to derive from the definition of $A(x)$ and Lemma \ref{l m(z0)-1 in H^11 and H^2}, we have $(A(\cdot)-1)\in L^{2,1}(\Real_+)\cap H^1(\Real_+)$ (note that $\im(z_1)>0$ is necessary). In addition, $(\det(A(\cdot))-1)\in L^{2,1}(\Real_+)\cap H^1(\Real_+)$. These two facts and \ref{e lower bound for det A} yield $B_j(\cdot)\in L^{2,1}(\Real_+)\cap H^1(\Real_+)$ for $j=1,2$. If we apply Corollary \ref{c u in H^11 and H^2 (positive half line)} to $v^{(0)}$ and $w^{(0)}$, we end up with $w^{(1)}\in H^{1,1}(\Real_+)$ and
    \begin{equation*}
        \partial_x \left(e^{-\frac{1}{2\ii} \int_{+\infty}^x|u^{(1)}(y)|^2dy } v^{(1)}_x(x)\right)\in L^2_x(\Real_+),
    \end{equation*}
    which is sufficient to conclude $u^{(1)}\in H^2(\Real_+)\cap H^{1,1}(\Real_+)$.
\end{proof}
\begin{cor}\label{c m^1(z0)-1 in H^11 and H^2}
   Let the assumptions of Lemma \ref{l solvability of RHP N=1} be valid and fix $z_2\in\Compl\setminus(\Real\cup \set{z_1,\overline{z}_1})$. Then for the solution $m^{(1)}(z;x)$ of \rh \ref{rhp m dynamic} we have $m^{(1)}(z_2;\cdot)-1\in H^1(\Real_+)\cap L^{2,1}(\Real_+)$. Moreover, if $\|r^{(1)}_+\|_{H^1\cap L^{2,1}}+\|r^{(1)}_-\|_{H^1\cap L^{2,1}}\leq M$ for some fixed $M>0$, then we also have the bound
   \begin{equation}\label{e m^1(z0)-1 in H^1 and L^2,1}
       \|m^{(1)}(z_2;\cdot)-1\|_{H^1(\Real_+)\cap L^{2,1}(\Real_+)}\leq C_M,
   \end{equation}
   where the constant $C_M>0$ depends on $M$, $z_1$, $z_2$ and $|c_1|$ but not on $r^{(1)}_{\pm}$.
\end{cor}
\begin{proof}
    (\ref{e Bäcklund for m^1}) can be written as
    \begin{multline*}
        m^{(1)}(z_2;x)=m^{(0)}(z_2;x)\\-2\ii\im(z_1) A(x)
        \left[
          \begin{array}{cc}
            0 & \frac{a_{21}(x)[m(z_2;x)]_{11}- a_{11}(x)[m(z_2;x)]_{21}} {(z_2-z_1)\det(A(x))} \\
            \frac{a_{22}(x)[m(z_2;x)]_{12}- a_{12}(x)[m(z_2;x)]_{22}} {(z_2-\overline{z}_1)\det(A(x))} & 0 \\
          \end{array}
        \right].
    \end{multline*}
    $m^{(1)}(z_2;\cdot)-1\in H^1(\Real_+)\cap L^{2,1}(\Real_+)$ is now a direct consequence of $m^{(0)}(z_2;\cdot)-1\in H^1(\Real_+)\cap L^{2,1}(\Real_+)$ (see Lemma \ref{l m(z0)-1 in H^11 and H^2}), $(A(\cdot)-1)\in L^{2,1}(\Real_+)\cap H^1(\Real_+)$, $(\det(A(\cdot))-1)\in L^{2,1}(\Real_+)\cap H^1(\Real_+)$ and (\ref{e lower bound for det A}).
\end{proof} 
\subsection{B\"{a}cklund transformation for $x<0$}
We consider the solution $m^{(1)}(z;x)$ of \rh \ref{rhp m^1} provided by Lemma \ref{l solvability of RHP N=1} and define
\begin{equation}\label{e def m^1 delta}
    m^{(1)}_{\delta}(z;x):=m^{(1)}(z;x)
    \left[
      \begin{array}{cc}
        \frac{z-z_1}{z-\overline{z}_1} & 0 \\
        0 & \frac{z-\overline{z}_1}{z-z_1} \\
      \end{array}
    \right]
    \left[
      \begin{array}{cc}
        \delta^{-1}(z) & 0 \\
        0 & \delta(z) \\
      \end{array}
    \right].
\end{equation}
The factor $\left(\frac{z-z_1}{z-\overline{z}_1}\right)^{\sigma_1}$ swaps the columns where the poles arise. The second factor $\delta^{-\sigma_1}$ has influence on the structure of the jump matrix. It can be shown by elementary calculations that (\ref{e def m^1 delta}) yields a solution of the following \rh.
\begin{samepage}
\begin{framed}
    \begin{rhp}\label{rhp m^1 delta}
        Find for each $x\in\Real$ a $2\times 2$-matrix valued function $\Compl\ni z\mapsto m_{\delta}^{(1)}(z;x)$ which satisfies
        \begin{enumerate}[(i)]
          \item $m_{\delta}^{(1)}(z;x)$ is meromorphic in $\Compl\setminus\Real$ (with respect to the parameter $z$).
          \item $m_{\delta}^{(1)}(z;x)=1+\mathcal{O}\left(\frac{1}{z}\right)$ as $|z|\to\infty$.
          \item The non-tangential boundary values $m^{(1)}_{\pm,\delta}(z;x)$ exist for $z\in\Real$ and satisfy the jump relation
              \begin{equation}\label{e jump m^1 delta}
                  m^{(1)}_{+,\delta}=m^{(1)}_{-,\delta} (1+R^{(1)}_{\delta}), \quad\text{where}\quad
                  R^{(1)}_{\delta}(z;x):=
                  \left[
                    \begin{array}{cc}
                       0 & e^{-2\ii zx}\overline{r}^{(1)}_{+,\delta}(z) \\
                       e^{2\ii zx}r^{(1)}_{-,\delta}(z) & \overline{r}^{(1)}_{+,\delta}(z) r_{-,\delta}^{(1)}(z) \\
                    \end{array}
                  \right],
              \end{equation}
              and $r^{(1)}_{\pm,\delta}(z):=r^{(1)}_{\pm} (z)\overline{\delta}_+(z) \overline{\delta}_-(z)\left( \frac{z-z_1}{z-\overline{z}_1}\right)^2$.
          \item $m_{\delta}^{(1)}$ has simple poles at $z_1$ and $\overline{z}_1$ with
              \begin{equation*}\label{e Res m^1 delta}
                  \begin{aligned}
                     \res_{z=z_1} m_{\delta}^{(1)}(z;x)&=\lim_{z\to z_1}m_{\delta}^{(1)}(z;x)
                  \left[
                      \begin{array}{cc}
                        0 & \frac{-e^{-2\ii z_1x} [2 \im(z_1)]^2} {\delta^{-2}(z_1)2\ii\lambda_1c_1} \\
                        0 & 0
                      \end{array}
                  \right],\\
                     \res_{z=\overline{z}_1} m_{\delta}^{(1)}(z;x)&=\lim_{z\to \overline{z}_1}m_{\delta}^{(1)}(z;x)
                  \left[
                      \begin{array}{cc}
                      0 & 0 \\
                      \frac{2\ii\lambda_1[2 \im(z_1)]^2e^{2\ii \overline{z}_1x}} {\delta^{2}(\overline{z}_1) \overline{c}_1} & 0
                      \end{array}
                  \right].
                  \end{aligned}
              \end{equation*}
        \end{enumerate}
    \end{rhp}
\end{framed}
\end{samepage}
Analogously to the previous subsection we set
\begin{equation*}
    r^{(0)}_{\pm,\delta}(z):= r^{(1 )}_{\pm,\delta}(z) \frac{z-\overline{z}_1}{z-z_1}= r^{(1)}_{\pm} (z)\overline{\delta}_+(z) \overline{\delta}_-(z) \frac{z-z_1}{z-\overline{z}_1},
\end{equation*}
and define $m^{(0)}_{\delta}(z;x)$ to be the unique solution
of \rh\ref{rhp m^0 delta} with data $\mathcal{S}^{(0)}_{\delta}:=\set{r^{(0)}_{\pm,\delta}}$. We have $r^{(0)}_{\pm,\delta}\in L^{2,1}\cap H^1$ and hence, the statements of Lemma \ref{l m(z0)-1 in in H^1 and L^2,1 (delta)} and Corollary \ref{c u in H^11 and H^2 (negative half line)} are available. Next we want to describe how the solutions $m^{(1)}_{\delta}(z;x)$ and $m^{(0)}_{\delta}(z;x)$ are connected by a B\"{a}cklund transformation of the form (\ref{e Bäcklund for m^1}). For this purpose we define
\begin{align*}
        &\left(
          \begin{array}{c}
            \vspace{1mm} a^{(\delta)}_{11}(x) \\
            a^{(\delta)}_{21}(x) \\
          \end{array}
        \right)
        :=m^{(0)}_{\delta}(\overline{z}_1;x)
        \left(
          \begin{array}{c}
            1 \\
            \frac{2\ii\lambda_1[2 \im(z_1)]^2e^{2\ii \overline{z}_1x}} {\delta^{2}(\overline{z}_1) \overline{c}_1} \\
          \end{array}
        \right),\\
        &\left(
          \begin{array}{c}
            \vspace{1mm} a_{12}^{(\delta)}(x) \\
            a_{22}^{(\delta)}(x) \\
          \end{array}
        \right)
        :=m^{(0)}_{\delta}(z_1;x)
        \left(
          \begin{array}{c}
            \frac{-e^{-2\ii z_1x} [2 \im(z_1)]^2} {\delta^{-2}(z_1)2\ii\lambda_1c_1} \\
            1 \\
          \end{array}
        \right),
\end{align*}
and $A^{(\delta)}(x)=
    \left[
      \begin{array}{cc}
        a^{(\delta)}_{11}(x) & a^{(\delta)}_{12}(x) \\
        a^{(\delta)}_{21}(x) & a^{(\delta)}_{22}(x) \\
      \end{array}
    \right]$. It turns out that
\begin{equation}\label{e Bäcklund for m^1 delta}
    m^{(1)}_{\delta}(z;x)=A^{(\delta)}(x)
    \left[
      \begin{array}{cc}
        z-\overline{z}_1 & 0 \\
        0 & z-z_1 \\
      \end{array}
    \right]
    \left[A^{(\delta)}(x)\right]^{-1}m^{(0)}_{\delta}(z;x)
    \left[
      \begin{array}{cc}
        \frac{1}{z-\overline{z}_1} & 0 \\
        0 & \frac{1}{z-z_1} \\
      \end{array}
    \right].
\end{equation}
Due to $\im(z_1)>0$ we have $e^{-2\ii z_1x}\in H^{1}_x(\Real_-)\cap L^{2,1}_x(\Real_-)$. Additionally, $(m^{(0)}_{\delta}(z_1;\cdot)-1)\in H^1(\Real_-)\cap L^{2,1}(\Real_-)$ and thus we find $(A^{(\delta)}(\cdot)-1)\in H^1(\Real_-)\cap L^{2,1}(\Real_-)$. These observation bring us in the position to extend the results of the previous subsection to the negative half-line.
\begin{cor}\label{c m^1(z0)-1 in L^21 and H^1 (delta)}
   Let the assumptions of Lemma \ref{l solvability of RHP N=1} be valid and fix $z_2\in\Compl\setminus(\Real\cup \set{z_1,\overline{z}_1})$. Then for the solution $m_{\delta}^{(1)}(z;x)$ of \rh\ref{rhp m^1 delta} we have $m_{\delta}^{(1)}(z_2;\cdot)-1\in H^1(\Real_-)\cap L^{2,1}(\Real_-)$. Moreover, if $\|r^{(1)}_+\|_{H^1\cap L^{2,1}}+\|r^{(1)}_-\|_{H^1\cap L^{2,1}}+|c_1|\leq M$ for some fixed $M>0$, then we also have the bound
   \begin{equation}\label{e m^1_delta(z0)-1 in H^1 and L^2,1}
       \|m_{\delta}^{(1)}(z_2;\cdot)-1\|_{H^1(\Real_-)\cap L^{2,1}(\Real_-)}\leq C_M,
   \end{equation}
   where the constant $C_M>0$ depends on $M$, $z_1$ and $z_2$ but not on $r^{(1)}_{\pm}$ and $|c_1|$.
\end{cor}
\begin{cor}\label{c u in H^11 and H^2 (negative half line - 1 pole)}
   Under the assumptions of Lemma \ref{l solvability of RHP N=1} the potential $u_{\delta}^{(1)}(x)$ reconstructed from the solution $m_{\delta}^{(1)}(z;x)$ of \rh\ref{rhp m^1 delta} by using (\ref{e rec 1}) and (\ref{e rec 2}) lies in $H^2(\Real_-)\cap H^{1,1}(\Real_-)$. Moreover, if $\|r^{(1)}_+\|_{H^1\cap L^{2,1}}+\|r^{(1)}_-\|_{H^1\cap L^{2,1}}+|c_1|\leq M$ for some fixed $M>0$, then we also have the bound
   \begin{equation}\label{e u^1 in H^11 and H^2 (negative half line)}
       \|u_{\delta}^{(1)}\|_{H^2(\Real_-)\cap H^{1,1}(\Real_-)}\leq C_M
   \end{equation}
   where the constant $C_M>0$ depends on $M$ and $z_1$, but not on $r^{(1)}_{\pm}$ and $|c_1|$.
\end{cor}
We finish the section with the following observation which is obvious from the definition:
\begin{equation*}
    \begin{aligned}
        \lim_{|z|\to\infty}z\;[m^{(1)}(z;x)]_{12}= \lim_{|z|\to\infty}z\;[m^{(1)}_{\delta}(z;x)]_{12},\\
        \lim_{|z|\to\infty}z\;[m^{(1)}(z;x)]_{21}= \lim_{|z|\to\infty}z\;[m^{(1)}_{\delta}(z;x)]_{21}.
    \end{aligned}
\end{equation*}
It follows that $u^{(1)}_{\delta}=u^{(1)}$. In conclusion the Corollaries \ref{c u in H^11 and H^2 (positive half line - 1 pole)} and \ref{c u in H^11 and H^2 (negative half line - 1 pole)} yield the existence of the mapping
\begin{equation}\label{e u in H^11 and H^2}
    H^1(\Real)\cap L^{2,1}(\Real)\ni (r^{(1)}_-,r^{(1)}_+)\mapsto u^{(1)}\in H^2(\Real)\cap H^{1,1}(\Real).
\end{equation}
 
\section{Proof of Theorem \ref{t main} }\label{s proof}
Inductively we can add more and more poles to the \rh \ref{rhp m}. Using the B\"{a}cklund transformation for $x>0$ and $x<0$ as described in the previous section we are able to show the following Lemma.
\begin{lem}\label{l solvability of RHP N geq 2}
    For any functions $r_{\pm}\in H^1(\Real)\cap L^{2,1}(\Real)$ which satisfy (\ref{e relation r+ r-}) - (\ref{e r constraint}), for any pairwise distinct poles $\lambda_1,...,\lambda_N$ with $\im(\lambda_k^2)>0$ and for any nonzero constants $c_1,...,c_N$, the \rh \ref{rhp m} is solvable. Moreover the function $u$ which can be obtained from $m$ by using (\ref{e rec 1}) and (\ref{e rec 2}) lies in $H^2(\Real)\cap H^{1,1}(\Real)$. If, in addition, $\|r^{(1)}_+\|_{H^1\cap L^{2,1}}+\|r^{(1)}_-\|_{H^1\cap L^{2,1}}+|c_1|+...+|c_N|\leq M$ for some fixed $M>0$, then we also have the bound
    \begin{equation}\label{e u in H^2 and H^11 N=2,3,4...}
        \|u\|_{H^2(\Real)\cap H^{1,1}(\Real)}\leq C_M
    \end{equation}
    where the constant $C_M>0$ is depending on $M$ and $\lambda_k$ but not on $r_{\pm}$ and $c_k$.
\end{lem}
Now we argue analogously to \cite{Pelinovsky2016} and assume that a local solution $u(\cdot,t)\in H^2(\Real)\cap H^{1,1}(\Real)\cap\pazocal{G}$ provided by the results in \cite{Tsutsumi1980} and \cite{Hayashi1992} blows up in a finite time. That is
$$
\lim_{t\uparrow T_{\text{max}}}\|u(\cdot,t)\|_{H^2(\Real)\cap H^{1,1}(\Real)}=\infty
$$
for a maximal existence time $T_{\text{max}}>0$. By (\ref{e u in H^2 and H^11 N=2,3,4...}) we conclude
$$
\lim_{t\uparrow T_{\text{max}}}\left[\|r_{+}(\cdot,t)\|_{H^1(\Real)\cap L^{2,1}(\Real)}+\|r_{-}(\cdot,t)\|_{H^1(\Real)\cap L^{2,1}(\Real)}+\sum_{k=1}^N|c_k|\right]=\infty,
$$
which contradicts the time evolution of the reflection coefficient and of the norming constants given in Lemma \ref{l time dependence scattering data}. This argument yields the proof of Theorem \ref{t main}.

\bibliographystyle{alpha}
\bibliography{lit}
\end{document}